\documentclass[11pt, hidelinks]{amsart}
\usepackage{writepack}
\usepackage{blindtext}  
\setlist{nosep}

\nocite{MAGMA}

\definecolor{myblue}{RGB}{37,150,190}
\definecolor{myred}{RGB}{204,87,23}

\title{Random endomorphisms of spherical reflection groups} 

\author{Isabelle Steinmann}

\date{}

\begin{document}

\begin{abstract}
    The goal of this paper is to understand the set $\End(W)$ of endomorphisms of an irreducible spherical reflection group $W$. We do this in two ways: numerically, by deriving an explicit formula for $|\End(W)|$; and probabilistically, by exploring the question \textit{what does a random endomorphism of $W$ look like?}
    For example, we show that as $n\to\infty$ the probability that a random endomorphism of $W_n$ is an automorphism tends to $\frac{1}{2}$ if $W_n=C_{2n}$ or $D_n$, to $\frac{1}{4}$ if $W_n=C_{2n+1}$, and to $1$ if $W_n=A_n.$
\end{abstract}

\maketitle

\section{Introduction}
\label{sec: intro}

    The irreducible spherical reflection groups were classified by Coxeter in the 1930s {\rm\cite{coxeter_1935}. They consist of four infinite families of groups, along with six exceptional groups. In this paper, we characterise the endomorphisms of an irreducible spherical group $W$ by computing $|\End(W)|$, and by analysing a random endomorphism of $W$.

\noindent
{\bf Counting endomorphisms.} For a group $G$, let 
    $$\mathcal{H}(G)=|\End(G)|$$
    be the size of the set of homomorphisms $\varphi:G\to G$.

    \begin{thm} The number of endomorphisms of an irreducible spherical reflection group is:
    \begin{enumerate}[label=(\alph*)]
        \item For $I_2(m)\cong \FD_m$, the order $2m$ symmetry group of the $m$-gon, for any $m\geq2$:
        $$\mathcal{H}(I_2(m))=
            \begin{cases}
                m^2+1 & \text{$m$ odd,}\\
                m^2+4m+4 & \text{$m$ even.}
            \end{cases}$$
           
        \item For $A_n\cong \FS_{n+1}$, the order $(n+1)!$ symmetry group of the $n$-simplex, when $n>3$ and $n\neq 5$:
            $$\mathcal{H}(A_n)=(n+1)!+\sum_{k=0}^{\lfloor\frac{n+1}{2}\rfloor} 
        \frac{(n+1)!}{2^k\;k!\;(n+1-2k)!}.$$
        
        \item  For $C_n\cong (\FC_2)^n\rtimes \FS_n$, the order $2^nn!$ symmetry group of the $n$-hypercube, when $n>2$ and $n\neq 4,6$:
            $$\mathcal{H}(C_n)= 2^{n}n!\Bigg(4+\sum_{k=0}^{\lfloor \frac{n}{4}\rfloor} \sum_{l=0}^{\lfloor \frac{n}{2}\rfloor-2k} \frac{(\frac{3}{8})^l \ 2^{n}}{2^{7k}\ k!\ l!\ (n-4k-2l)!}\Bigg) .$$    

        \item For $D_n\cong (\FC_2)^{n-1}\rtimes \FS_n$, the order $2^{n-1}n!$ index $2$ subgroup of $C_n$ consisting of signed permutations with an even number of coordinate inversions, when $n>2$ and $n\neq 4,6:$
            $$
             \mathcal{H}(D_n)=2^{n+\frac{-1+(-1)^n}{2}}n!+
            \frac{n!}{\lfloor\frac{n}{2}\rfloor}+\sum_{k=0}^{\lfloor \frac{n}{2}\rfloor-1} \frac{2^{n-1}n!}{2^{2k}\ k!\ (n-2k)!}.$$

        \item For the exceptional irreducible spherical reflection groups, and the members of the infinite families omitted above:
        \begin{table}[h]
            \centering
            \begin{tabular}{l l l}
                 $\CH(H_3)=272$ & $\CH(H_4)=29372$& $\CH(F_4)=30880$\\ 
                 $\CH(E_6)=52732$& $\CH(E_7)=2913248$& $\CH(E_8)=696929552$\\  
                 $\CH(A_3)=58$& $\CH(A_5)=1516$& $\CH(C_4)=6496$\\
                 $\CH(C_6)=476416$& $\CH(D_4)=3116$& $\CH(D_6)=138992$
            \end{tabular}
        \end{table}
    \end{enumerate}
    \label{thm: main theorem}
    \end{thm}

   The description of the set of endomorphisms of irreducible spherical reflection groups gives us additional information with which we can enumerate homomorphisms between these groups. As a proof of concept, the following theorem  computes the orders of $\Hom(I_2(p),C_n)$ and $\Hom(I_2(p),A_{n})$ where $p$ is an odd prime. 
   
   \begin{thm}
        Suppose $p$ is an odd prime and $W_n$ is $C_n$ or $A_{n-1}\cong \FS_n$. Then the number of homomorphisms $I_2(p)\to W_n$ is
            \begin{equation*}
            |\mathrm{Hom}(I_2(p),W_n)|=
                \sum_{k=0}^{\lfloor\frac{n}{p}\rfloor}
                \sum_{l=0}^{\lfloor\frac{k}{2}\rfloor}
                \sum_{m=0}^{\lfloor\frac{n-kp}{2}\rfloor}
                \frac{|W_n|}{(k-2l)!\;l!\;p^l\;2^{\tau(l+m)}\;(n-kp-2m)!\; m!}
            \end{equation*}
       where  $\tau=2$ when $W_n=C_n$, and $\tau=1$ when $W_n=A_{n-1}$.
       \label{thm: hom(I2(p),W)}
   \end{thm}

    It would be interesting to extend this Theorem to finding $|\Hom(W,W')|$ for other irreducible spherical reflection groups $W$ and $W'$.

\noindent
{\bf Random endomorphisms.} The proof of Theorem \ref{thm: main theorem}  does more than compute $|\End(W)|$, it also finds the endomorphism table for $W$ which describes the endomorphisms of $W$ in terms of their kernels and images (see Definition \ref{def: endomorphism table}). Hence, it allows us to study properties of endomorphisms of groups to try and answer the question:
    
    \leftskip1.5cm\relax
    \rightskip1.5cm\relax
    \begin{center}
    \textit{What does a random endomorphism of a spherical reflection group look like?} \\ 
    \end{center}
    \leftskip0cm\relax
    \rightskip0cm\relax
    
    For example, we can ask how large the image of a random endomorphism is, or whether the centre of the group is contained in the kernel of a random endomorphism. 
    Recall that the groups $C_n$, $D_n$, and $A_{n}$ have orders $2^nn!$, $2^{n-1}n!$, and $(n+1)!$, respectively. 

    \begin{thm}
        Suppose $W_n$ is $C_n$, $D_n$, or $A_{n}$. Then most of the endomorphisms of $W_n$ have large images: asymptotically, the expected order of the image of a random $\varphi\in \End(W_n)$ is:
        $$  \EE_{\varphi\in \End(W_n)} \big[\ |\varphi(W_n)|\  \big]\sim \begin{cases}
            2^{2n-1}(2n)! & W_n=C_{2n},\\
            3\cdot 2^{(2n+1)-3}(2n+1)! & W_n=C_{2n+1},\\
            2^{n-2}n! & W_n=D_n,\\
            (n+1)! & W_n=A_n.
        \end{cases} $$
        Moreover, the standard deviation of the order of the image is asymptotically
        $$  \sigma_{\varphi\in \End(W_n)} \big[\ |\varphi(W_n)|\  \big]\sim \begin{cases}
            2^{2n-1}(2n)! & W_n=C_{2n},\\
            \sqrt{11}\cdot 2^{(2n+1)-3}(2n+1)! & W_n=C_{2n+1},\\
            2^{n-2}n! & W_n=D_n,\\
            0 & W_n=A_n.
        \end{cases} $$
        \label{thm: expected kernel-index of Cn, Dn, An}
    \end{thm}

      Figure \ref{fig: proportion End(Cn)} exhibits trends in the order of the image of a random endomorphism of $C_n$. 
    \begin{figure}[h]
        \centering
        \includegraphics[width=0.92\linewidth]{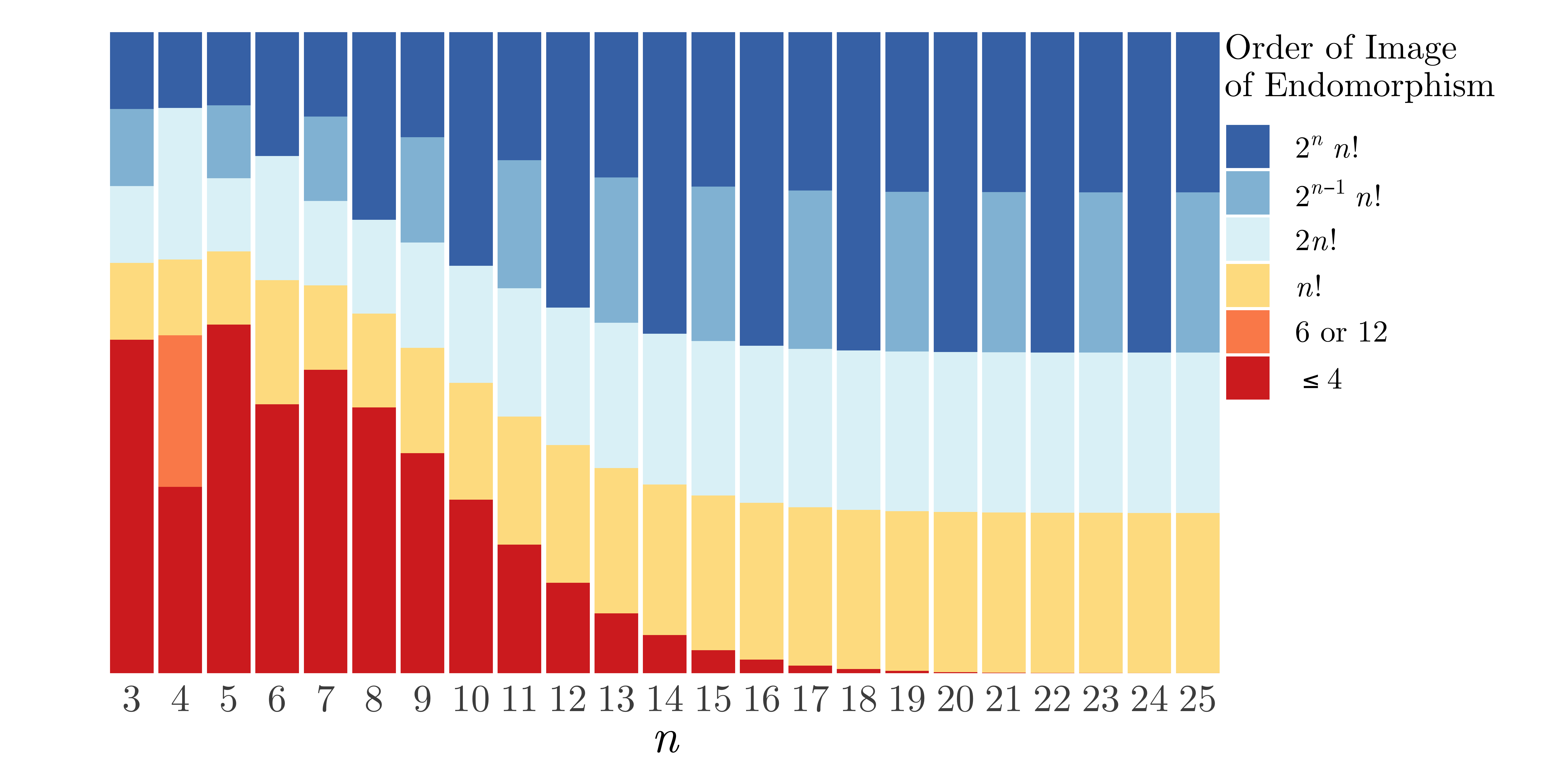}
        \caption{A stacked bar graph which shows the proportional orders of the images of endomorphisms of $C_n$ for $3\leq n \leq 25$.
        For example, when $n=5$ the graph shows that roughly half of the endomorphisms of $C_5$ have image with order less than $n!$. 
        However, we see that for a random $\varphi\in \End(C_n)$ the probability that $|\varphi(C_n)|<n!$ tends to $0$ as $n\to \infty$.
        Note that for $\varphi\in \End(C_n)$, if $|\varphi(C_n)|=2^nn!=|C_n|$, then $\varphi\in \Aut(C_n)$. So, as $n\to\infty$, half of the endomorphisms of $C_{2n}$ are automorphisms, and one quarter of the endomorphisms of $C_{2n+1}$ are automorphisms. \\
        There are exceptional endomorphisms of $C_n$ when $n$ is $4$ or $6$, and these occur because of the exceptional properties of $\FS_4$ and $\FS_6$. More specifically: the `extra' endomorphisms of $C_4$ with image-order $6$ and $12$ are due to the `extra' normal subgroup $V_4\unlhd \FS_4\leqslant C_4$; and the `extra' endomorphisms of $C_6$ with image-order $6!$ and $2\cdot 6!$ are induced by the nontrivial outer automorphism of $\FS_6$.
        }
        \label{fig: proportion End(Cn)}
    \end{figure}

    In addition to the expected image-order of a random endomorphism being large, it is also likely that a random endomorphism is, in fact, an automorphism, and hence has maximum image-order.
    
    \begin{thm}
        Suppose $W_n$ is $C_n$, $D_n$, or $A_{n}$. Then 
        \begin{align*}
        \lim_{n\to\infty} \PP_{\varphi\in \End(W_n)} \big[ \text{$\varphi$ is an automorphism} \big]=\begin{cases}
            \frac{1}{2} & W_n=C_{2n},D_n, \\ 
            \frac{1}{4} & W_n=C_{2n+1}, \\ 
            1 & W_n=A_n.
        \end{cases}
        \end{align*}
        \label{thm: prob of automorphism Cn, Dn, An}
    \end{thm}

    A natural question we can ask is: what is the probability that the generator of the centre of $W_n$ is contained in the kernel of a random endomorphism of $W_n$? The centres of $C_n$ and $D_{2n}$ have order $2$, whereas the centres of $A_n$ and $D_{2n+1}$ are trivial. 
    
    \begin{thm}
        As $n\to \infty$, half of the endomorphisms of $C_n$ or $D_{2n}$, and all of the endomorphisms of $D_{2n+1}$ or $A_n$ contain the centre in their kernel. 
        $$ \lim_{n\to \infty} \PP_{\varphi\in \End(W_n)}\big[ Z(W_n)\subseteq\ker(\varphi)\big]=\begin{cases}
            \frac{1}{2} & W_n=C_{n},D_{2n}, \\ 
            1 & W_n=D_{2n+1},A_n.
        \end{cases}$$
        \label{thm: central endo}
    \end{thm}
    
    Although the statement of Theorem \ref{thm: central endo} for $W_n=C_n$ does not depend on the parity of $n$, the way in which the probability approaches $\frac{1}{2}$ does. We demonstrate this in Figure \ref{fig: Prob of centre in kernel Cn}, which shows that the limit as $n\to \infty$ of the probability is decreasing for $C_{2n}$ and increasing for $C_{2n+1}$.
    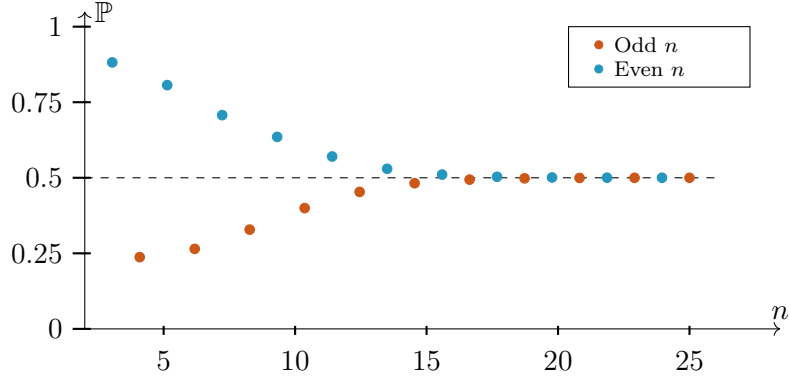
\begin{figure}[h]
        \centering
        \begin{tikzpicture}[x=8cm, y=4cm]          
          \draw[->] (0,0) -- (1.15,0) node[above] {$n$};
          \draw[->] (0,0) -- (0,1.05) node[right] {$\PP$};
        
          \foreach \i in {5,10,15,20,25} {
            \pgfmathsetmacro{\tick}{(\i-2)/23}
            \draw[thick] (\tick,0.01) -- (\tick,-0.02);
            \node[below] at (\tick,-0.04) {\i};
          }
    
           \foreach \y in {0,0.25,0.5,0.75,1} {
            \draw[thick] (0.01,\y) -- (-0.02,\y);
            \node[left] at (-0.02,\y) {\y};
          }
        
          \draw[black,thin,dashed] (0,0.5) -- (1.05,0.5);
        
          \foreach \i/\val [count=\j from 4] in {4/0.88177,5/0.23741,6/0.80656,7/0.26476,8/ 0.70728, 9/0.32847, 10/0.63536, 11/0.39972,12/ 0.57050,13/ 0.45320, 14/0.52972,15/ 0.48195, 16/0.51060,17/ 0.49398, 18/0.50333, 19/0.49820,20/ 0.50096, 21/0.49950, 22/0.50025, 23/0.49987, 24/0.50007, 25/0.49997
          } {
            \pgfmathsetmacro{\x}{(\i-3)/22}
            \pgfmathsetmacro{\mydotcolor}{isodd(\i) ? "myred" : "myblue"}
            \filldraw[\mydotcolor] (\x,\val) circle [radius=1.8pt];
            }
        \begin{scope}[shift={(0.8, 0.8)}]
            \draw[fill=white] (0,0) rectangle (0.3, 0.2);
            
            \filldraw[myred] (0.05, 0.14) circle (1.5pt) 
                node[right, black, font=\scriptsize, xshift=2pt] {Odd $n$};
                
            \filldraw[myblue] (0.05, 0.06) circle (1.5pt) 
                node[right, black, font=\scriptsize, xshift=2pt] {Even $n$};
        \end{scope}
        \end{tikzpicture}          
    \caption{The probability that the centre of $C_n$ lies in the kernel of a random endomorphism of $C_n$ for $4\leq n \leq 25$. As $n$ tends to infinity, the sequence of probabilities tends to $\frac{1}{2}$. For even $n$, the sequence of probabilities is decreasing, while for odd $n$ it is increasing.}
    \label{fig: Prob of centre in kernel Cn}
    \end{figure}

    Theorem \ref{thm: hom(I2(p),W)}, which computes the number of homomorphisms from $I_2(p)$ to $C_n$ or $\FS_n$ when $p$ is an odd prime, allows us to plot $|\Hom(I_2(p),C_n)|$ for different values of $p$ and $n$ in  Figure \ref{fig: Hom(I2(m),Cn)}.
    \begin{figure}[h]
        \centering
        \includegraphics[width=0.98\linewidth]{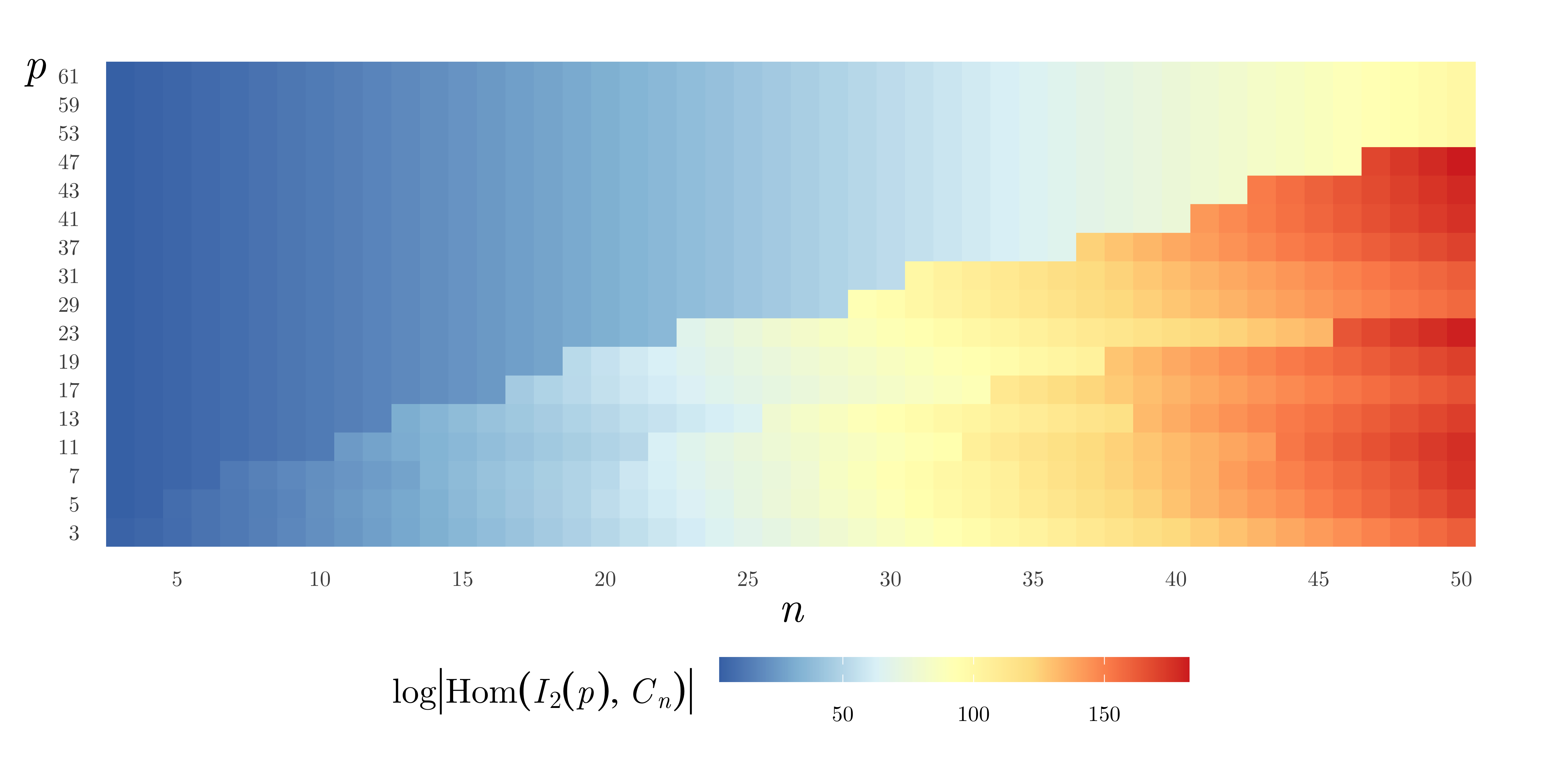}
        \caption{The log-transformed order of the set $\Hom(I_2(p),C_n)$, where $p$ is an odd prime satisfying $3\leq p\leq 61$ and $3\leq n \leq 50$.         
        When $p>n$ there are no monomorphisms $I_2(p)\hookrightarrow C_n$ and so $|\mathrm{Hom}(I_2(p),C_n)|$  does not depend on $p$. This results in the vertical lines in the upper left portion of the figure.        
        For a fixed $p$, with each increase in $k=\lfloor \frac{n}{p} \rfloor$, there is an increase in the number of conjugacy classes of order-$p$ elements in $C_n$, which leads to a jump in the number of monomorphisms $I_2(p)\hookrightarrow C_n$, and in
        $|\mathrm{Hom}(I_2(p),C_n)|$. For example, we see a sharp increase in $\log|\Hom(I_2(13),C_n)|$ when $n=13$, $26$, and $39$. The order of $\Hom(I_2(p),C_n)$ grows relatively steadily between multiples of $p$, and this is shown in the horizontal gradient. The graph has some `pale streaks' (at $p=23$, for example) because of variations in the gaps between subsequent primes. }
        \label{fig: Hom(I2(m),Cn)}
    \end{figure}

\noindent
{\bf Methods.} There is a bijection between $\Hom(G,H)$ and the set of tuples $(K,M,\varphi)$, where $K\unlhd G$ is normal, $M\leqslant H$ is isomorphic to $G/K$, and $\varphi$ is an automorphism of $G/K$. 
    Note that normal subgroups $K\unlhd G$ are merely \emph{potential} kernels because, if there is no subgroup of $H$ isomorphic to $G/K$, then there is no homomorphism from $G$ to $H$ with kernel $K$. For example, there are no subgroups of $C_n$ isomorphic to $C_n/Z(C_n)$ when $n$ is even, but two such subgroups when $n$ is odd. The normal subgroups of spherical reflection groups and the quotient groups they induce are characterised by Maxwell in {\rm\cite{maxwell_normal_subs}}, and the automorphisms of spherical reflection groups are well-known and are described in {\rm\cite[Ch. 2]{AutW_Franzsen}}.

    Every spherical reflection group $W$ has as an index $2$ subgroup, the \emph{rotation subgroup} $W^+$, which consists of all determinant $1$ elements of $W$. Hence for any $W$ we will be interested in subgroups of $W$ isomorphic to $W/W^+\cong \FC_2$. These subgroups are in bijection with $\mathrm{Inv}(W)$, the set of involutions of $W$. Let $\mathcal{V}(W)=|\mathrm{Inv}(W)|$ denote the number of involutions in $W$. 
    More generally, suppose $G$ and $H$ are groups and let $K$ be a normal subgroup of $G$. Define the following:
    \begin{align*}
            \mathcal{Z}(G,H,K)&:=|\{F\leqslant H:F\cong G/K\}|, \\ 
            \mathcal{E}(G,H,K)&:=|\{\varphi\in \Hom(G,H):K=\ker(\varphi)\}|.
    \end{align*}
    When describing endomorphisms of $G$ let $\mathcal{Z}(G,K):=\mathcal{Z}(G,G,K)$, and similarly let $\mathcal{E}(G,K):=\mathcal{E}(G,G,K)$. 
    Note that 
    \begin{equation*}
        \mathcal{E}(G,H,K)=\mathcal{Z}(G,H,K)\cdot|\Aut(G/K)|,
    \end{equation*}
    and moreover
    \begin{equation*}
        \big|\mathrm{Hom}(G,H)\big|=\sum_{K\unlhd G} \mathcal{E}(G,H,K)=\sum_{K\unlhd G} \mathcal{Z}(G,H,K)\cdot|\Aut(G/K)|.
    \end{equation*}

    \noindent
    {\bf Acknowledgements.}
    I would like to thank Benson Farb for his suggestion of this problem and for his guidance and encouragement. I would also like to thank Tai Lohrer for helping to create Figures \ref{fig: proportion End(Cn)} and \ref{fig: Hom(I2(m),Cn)}, and Sehyun Ji for his helpful suggestions in Section \ref{sec: random endos}.

\section{The endomorphism table of $I_2(m)$}
\label{sec: endos of I2(m)}

     \begin{defn}
    \label{def: endomorphism table}
    For any pair of groups $(G,H)$, one can construct a \emph{homomorphism table} (or \emph{endomorphism table} if $G= H$) of $(G,H)$ which collates data on the homomorphisms from $G$ to $H$. The rows of this table correspond to normal subgroups $K$ of $G$. The first two columns state $|G:K|$ and give a description of $K$. The third column describes the isomorphism type of the quotient group $G/K$, and hence the isomorphism type of the image of a homomorphism from $G$ to $H$ with kernel $K$. 
    The fourth gives $\mathcal{Z}(G,H,K)$, the number of subgroups of $H$ isomorphic to $G/K$. These are (setwise) images of homomorphisms from $G$ to $H$ with kernel $K$. The fifth lists the order of the automorphism group of $G/K$. This counts the number of homomorphisms from $G$ to $H$ with kernel $K$ and (setwise) fixed image. The final column counts the total number of homomorphisms from $G$ to $H$ with kernel $K$, and it is obtained by multiplying the entries of the previous two columns. For an example, see Table \ref{tab: counting End(I2(m))} on page 6.
    \end{defn}
    
    The homomorphism table of $(G,H)$ not only allows us to compute $|\Hom(G,H)|$ (by summing the entries in the final column), but also gives a classification of homomorphisms from $G$ to $H$ in terms of their kernels. Note that $\mathcal{Z}(G,H,\{1\})$ counts the number of subgroups of $H$ isomorphic to $G$, and $\mathcal{E}(G,H,\{1\})$ counts the number of injections of $G$ into $H$. Conversely, if there is some $K\unlhd G$ such that $G/K\cong H$, then $\mathcal{E}(G,H,K)$ counts the number of surjections of $G$ onto $H$.
    
    The spherical reflection group $I_2(m)$ is isomorphic to the dihedral group $\mathfrak{D}_m$ of order $2m$ and so it has the following presentation
    $$I_2(m) =\lb r,s \mid r^m=s^2=(rs)^2=1\rb .$$
    The automorphism group of $I_2(m)$ has order $m\phi(m)$ when $m\geq3$, order $6$ when $m=2$, and order $1$ when $m=1$ {\rm\cite[p. 33]{AutW_Franzsen}}. 
    Since every group generated by two involutions is dihedral, if $N$ is an index $k>1$ normal subgroup of $I_2(m)$, then $I_2(m)/N\cong I_2(k/2)$. In particular, every nontrivial normal subgroup of $I_2(m)$ has even index. 

    We will briefly describe the subgroups and normal subgroups of $I_2(m)$ in terms of $r$ and $s$; more information on this can be found in {\rm\cite{Cavior_dihedral,Conrad_dihedral}}. Suppose $k$ divides $2m$. If $k$ is odd, then $I_2(m)$ has $k$ subgroups of index $k$, all of which are conjugate to $\lb r^k,s\rb\cong I_2(m/k)$, so these are normal if and only if $k=1$. If $k$ is even and does not divide $m$, then $I_2(m)$ has a unique (and therefore normal) index $k$ subgroup which is cyclically generated by $r^{k/2}$. Finally, if $k$ is even and divides $m$, then $I_2(m)$ has $k+1$ index $k$ subgroups. There is the normal cyclic subgroup generated by $r^{k/2}$, as well as $k$ subgroups isomorphic to $I_2(m/k)$, half of which are conjugate to $\lb r^k,s\rb$, and half of which are conjugate to $\lb r^k,rs\rb$. These $k$ subgroups are normal if and only if $k=2$. 

    \begin{prop}
        \label{prop: H(I2(m))}
    Table \ref{tab: counting End(I2(m))} is the endomorphism table of $I_2(m)$, and the number of endomorphisms of $I_2(m)$ is
    $$ \mathcal{H}(I_2(m))=\begin{cases}
            m^2+1 & \text{$m$ odd ,}\\
            m^2+4m+4 & \text{$m $ even.}
        \end{cases} $$
        \end{prop}
    \begin{table}[h]
        \begin{adjustbox}{width=\columnwidth,center}
        \begin{tabular}{c c c c c c}
        \hline
             $|I_2(m):K|$ & $K\unlhd I_2(m)$ &  $I_2(m)/K$ & $\mathcal{Z}(I_2(m),K)$ & $|\Aut(I_2(m)/K)|$ & $\mathcal{E}(I_2(m),K)$ \\ 
        \hline
        
            $1$ & $I_2(m)$& $\{1\}$ & $1$ & $1$ & $1$ \\ 

            $2$ ($m$ even) & $\lb r^2,s\rb$ & $I_2(1)$ & $m+1$ & $1$ & $m+1$ \\

            $2$ ($m$ even) & $\lb r^2,rs\rb$ & $I_2(1)$ & $m+1$ & $1$ & $m+1$ \\

            $2d$ $(d\mid m)$ & $\lb r^d\rb$ & $I_2(d)$ & 
            $\begin{cases}
                m+1 & \text{$m$ even, $d=1$,}\\
                m/d & \text{else.}
            \end{cases}$ & 
            $\begin{cases}
                6 & d=2,\\
                \phi(d)d & \text{else.}
            \end{cases}$ & 
            $\begin{cases}
                m+1 & \text{$m$ even, $d=1$,}\\
                3m & d=2,\\
                \phi(d)m & \text{else.}
            \end{cases}$\\
            
        \hline 
        \end{tabular}
        \end{adjustbox}
        \caption{The endomorphism table of $I_2(m)$. See Definition \ref{def: endomorphism table} for a description of an endomorphism table.}
        \label{tab: counting End(I2(m))}
    \end{table}
    \begin{proof}
        Let $m$ be odd. Then, aside from the trivial map, every endomorphism of $I_2(m)$ has as its kernel the index $2d$ subgroup $\lb r^d\rb\cong \FC_{m/d} $ for some $d\mid m$. 
        The image of such an endomorphism is isomorphic to $I_2(d)$ and so has index $m/d$ in $I_2(m)$. This number is odd, so there are $m/d$ subgroups of $I_2(m)$ isomorphic to $I_2(d)$, and the automorphism group of $I_2(d)$ has order $\phi(d)d.$ This allows us to compute 
        $$ \mathcal{H}(I_2(m))=1+\sum_{d\mid m} m\phi(d)=1+m^2.
        $$
        A similar argument applies for even $m$. 
    \end{proof}

\section{The endomorphism tables of $C_n$, $D_n$, and $A_n$}
\label{sec: end(Cn), end(Dn), end(An)}

    The symmetry group of the $n$-cube acts on $\RR^n$ by permuting and inverting an orthonormal basis $e_1,\dots,e_n$ of $\RR^n$, and is isomorphic to the irreducible spherical reflection group of type $C_n$. 
    Let $r_{+1,i,j}$ be the reflection with root $e_i-e_j$ which permutes the $i$th and $j$th coordinates, $r_{-1,i,j}$ the reflection with root $e_i+e_j$ which permutes \textit{and} inverts the $i$th and $j$th coordinates, and $r_{i}$ the reflection with root $e_i$ which inverts the $i$th coordinate.    

    The subgroup of coordinate inversions, $N:=\lb r_1,\dots,r_n\rb \cong (\FC_2)^n$, is normal in $C_n$, and 
    $$C_n=N\rtimes S_+\cong \FC_2\wr \FS_n $$ 
    where $S_+:=\lb r_{+1,1,2},\dots, r_{+1,n-1,n}\rb \cong \mathfrak{S}_n$ is the subgroup of coordinate permutations.
    Since any $w\in C_n$ can be uniquely written as $w=x\sig$ where $x\in N$ and $\sig\in S_+$, we can define the \emph{sign} of $w$ to be the tuple $(\det(x),\det(\sigma))$. This encodes the parity of the number of reflections used to write $x$ and $\sig$.
    Let $\pi$ denote the quotient map $\pi:C_n\to\mathfrak{S}_n$  defined by $x\sig\mapsto \sig$. 
    The centre of $C_n$ is generated by the involution $z_0=r_1\dots r_n$, which acts on $\RR^n$ as multiplication by $-I_n$ {\rm\cite[p. 8]{AutW_Franzsen}}.  

    Following the notation of Maxwell {\rm\cite{maxwell_normal_subs}}, which relates a normal subgroup to the conjugacy class(es) of its generators, let $(C_n)_1$ be the index $2$ subgroup of $C_n$ consisting of isometries which involve an even number of coordinate-inversions. For $n\geq4$, this subgroup is isomorphic to the irreducible spherical reflection group of type $D_n$, and hence isomorphic to $(\FC_2)^{n-1}\rtimes \FS_n$. The subgroup $S_+\leq C_n$ of coordinate permutations is isomorphic to the irreducible spherical reflection group of type $A_{n-1}$. Consequently, we explore the endomorphisms of $D_n$ and $A_{n-1}$ concurrently with those of $C_n$.

    \begin{thm}{\rm\cite[p. 21,28]{AutW_Franzsen}}
        For $n\neq 4$, the automorphism group of $C_n$ has order $2^{n+1}n!=2|C_n|$ when $n$ is even; and $2^nn!=|C_n|$ when $n$ is odd. Similarly for $n\neq 4$, the automorphism group of $D_n$ has order $2^{n}n!=2|D_n|$ when $n$ is even; and $2^{n-1}n!=|D_n|$ when $n$ is odd. For $n\neq5$, the automorphism group of $A_n$ has order $(n+1)!=|A_n|$. 
        \label{thm: |Aut| for Cn, Dn, An}
    \end{thm}

    Let $w\in C_n$ and suppose $\pi(w)\in \mathfrak{S}_n$ has cycle-type $(\la_1,\dots,\la_l)$. We can decompose $w$ as $w_1\dots w_l$ so that $w_i$ induces a $\la_i$-cycle in $\mathfrak{S}_n$. We can also partition the orthonormal basis of $\RR^n$ into sets $P_1,\dots,P_l$ with $|P_i|=\la_i$ in such a way that the action of $w$ preserves this partition, and moreover each $P_i$ is pointwise fixed by the action of $w_j$ when $j\neq i$. 
    Then, for any $e_j\in P_i$, we have $w^{k_i}(e_j)=w_i^{k_i}(e_j)=\pm e_j$, and this sign is independent of the choice of $e_j$. We then say that $w_i$ has signed cycle-type $\pm \la_i$, and $w$ has \emph{signed cycle-type} given by the multiset $\{\pm \la_1,\dots,\pm \la_l\}$. For $w\in S_+\cong \FS_n$, the signed cycle type of $w$ coincides with the usual permutation cycle-type, and so determines the conjugacy class of $w$. 

    \begin{thm}[{\rm{\cite[p. 306]{carter_conjugacy_1972}}}]
        Two elements of $C_n$ are conjugate in $C_n$ if and only if they have the same signed cycle-type. Two elements in $D_n$ are conjugate in $D_n$ if and only if they have the same signed cycle-type, except when every signed cycle has positive sign and even length, in which case there are two conjugacy classes.
    \end{thm}
    
    If $w$ is an involution, its signed cycle-type consists of $+2$, $-1$, and  $+1$, with some multiplicities. Hence we can assign to $w$ an \emph{involution-type} $(t,u)$, where $t$ encodes the multiplicity of $-1$ in the signed cycle-type of $w$, and $u$ the multiplicity of $+2$. Note that $2t+u\leq n$, and the sign of an involution of involution-type $(t,u)$ is $((-1)^t,(-1)^u)$.

    \begin{exmps}
        The reflection $r_{\pm1,i,j}$ has signed cycle-type $\{+2,+1,\dots,+1\}$, involution-type $(0,1)$, and sign $(1,-1)$.
        The reflection $r_i$ has signed cycle-type $\{-1,+1,\dots,+1\}$, involution-type $(1,0)$ and sign $(-1,1)$.
        The central element $z_0=r_1\dots r_n$ has signed cycle-type $\{-1,\dots,-1\}$, involution-type $(n,0)$ and sign $((-1)^n,1)$.
        The element $r_1r_{+1,1,2}\in C_2$ has signed cycle-type $\{-2\}$ and sign $(-1,-1)$. 
    \end{exmps}

    Since $C_2$ is isomorphic to $I_2(4)$ and $A_2$ is isomorphic to $I_2(3)$, their respective $36$ and $10$ endomorphisms are described in Section \ref{sec: endos of I2(m)}. Hence we can assume for the remainder of this Section that $n$ is at least $3$.
    
    \begin{thm}[Normal subgroups of $C_n$, $D_n$, and $A_n$]{\rm\cite[p. 371]{maxwell_normal_subs}}
    \label{thm: normal subs of Cn, Dn and An}
    For $n>2$ and $n\neq4$, the group $C_n$ has seven nontrivial normal subgroups, and $C_4$ has nine nontrivial normal subgroups. For $n>2$ and $n\neq 4$, the group $D_n$ has two nontrivial normal subgroups when $n$ is odd, and three nontrivial normal subgroups when $n$ is even, and $D_4$ has six nontrivial normal subgroups. For $n>3$, the group $A_n$ has one nontrivial normal subgroup, and $A_3$ has two nontrivial normal subgroups. 
    \end{thm}

    We will now describe the nontrivial normal subgroups of $C_n$, $D_n$ and $A_n$, as well as the quotients of these groups by their normal subgroups.
    
    \begin{itemize}
        \item[$C_n$:] There are three index $2$ subgroups of $C_n$. These are $(C_n)_1$, $\prescript{}{+}{C_n}$, and $C_n^+$, and these subgroups consist of elements with sign $(+1,\pm1)$, $(\pm1,+1)$, and $\pm(+1,+1)$, respectively. The group $(C_n)_1$ is isomorphic to $(\FC_2)^{n-1}\rtimes S_n$, and hence to the irreducible spherical reflection group of type $D_n$ for $n\geq 4$. The group $\prescript{}{+}{C_n}$ is isomorphic to $(\FC_2)^n\rtimes \mathfrak{A}_n$.
        The quotient of $C_n$ by any of these three index $2$ subgroups is isomorphic to $\FC_2$ and so we compute $\mathcal{V}(C_n)$, the number of involutions in $C_n$, in Lemma \ref{lem: invs in Cn, Dn, An}. The unique index $4$ normal subgroup $(C_n)_1^+\leqslant C_n$ is composed of the elements of $C_n$ with sign $(+1,+1)$, and the quotient of $C_n$ by $(C_n)_1^+$ is isomorphic to $(\FC_2)^2$. Consequently, in Lemma \ref{lem: comm order leq 2} we compute the number of unordered pairs of commuting elements of $C_n$ with order at most 2, and we use this to find the number of subgroups of $C_n$ isomorphic to $(\FC_2)^2$ in Lemma \ref{lem: subs of Cn isom (Z/2Z)^2}. The unique index $n!$ and $2\cdot n!$ normal subgroups of $C_n$ are the elementary abelian groups $N$ and $N^+$, the subgroup of coordinate inversions, and the subgroup of even coordinate inversions (Maxwell {\rm\cite{maxwell_normal_subs}} also denotes these subgroups $(C_n)_2$ and $(C_n)_2^+$, respectively). The quotient of $C_n$ by $N$ is isomorphic to $\FS_n$ and hence we find the subgroups of $C_n$ isomorphic to $\FS_n$ in Lemma \ref{lem: subgroups isom to Sn in Cn and Dn}. Since $C_n/N^+$ is isomorphic to $\FC_2\times\FS_n$, we will count how many involutions centralise $S_+$ in Lemma \ref{lem: subgroups isom to Z/2ZxSn in Cn}. Finally, the centre of $C_n$, which we denote $\{\pm1\}$, is generated by the involution $z_0=r_1\dots r_n$, so any subgroup of $C_n$ isomorphic to $C_n/\{\pm1\}$ has index $2$ in $C_n$ and hence is one of $(C_n)_1$, $\prescript{}{+}{C_n}$, and $C_n^+$. In Lemma \ref{lem: index 2 in Cn} we will ascertain which of these index 2 subgroups is isomorphic to $C_n/\{\pm1\}$. 
        \item[$D_n$:] When $n$ is odd, the two nontrivial normal subgroups of $D_n$ are the index $2$ rotation subgroup $D_n^+$ (denoted $(C_n)_1^+$ when viewed as a subgroup of $C_n$), and the index $n!$ subgroup $N$ of even coordinate inversions (denoted $N^+$ when viewed as a subgroup of $C_n$). The quotient of $D_n$ by $D_n^+$ is isomorphic to $\FC_2$, so $\mathcal{V}(D_n)$ is computed in Lemma \ref{lem: invs in Cn, Dn, An}; and the quotient of $D_n$ by $N$ is isomorphic to $\FS_n$, so in Lemma \ref{lem: subgroups isom to Sn in Cn and Dn} we find the number of subgroups of $D_n$ of this kind. When $n>4$ is even, $D_n$ has one additional nontrivial normal subgroup, its centre $\{\pm1\}$ which, like for $C_n$, is generated by the involution $z_0=r_1\dots r_n$. Hence in Lemma \ref{lem: index 2 in Cn}, we will check whether $D_n/\{\pm1\}$ is isomorphic to $D_n^+$.
        \item[$A_n$:]  Since the group $A_{n-1}$ is isomorphic to $\FS_n$, when $n\geq5$ it has only one nontrivial normal subgroup, the rotation subgroup $A_{n-1}^+$ (note that $A_{n-1}^+$ is isomorphic to the alternating subgroup $\mathfrak{A}_n$). The number of involutions in $A_{n-1}$ is calculated in Lemma \ref{lem: invs in Cn, Dn, An}.
    \end{itemize}

    In summary, we need to count the involutions in $C_n$, $D_n$ and $A_n$, the subgroups of $C_n$ isomorphic to $(\FC_2)^2$, the subgroups of $C_n$ and $D_n$ isomorphic to $\FS_n$, the subgroups of $C_n$ isomorphic to $\FC_2\times \FS_n$, the subgroups of $C_n$ isomorphic to $C_n/\{\pm1\}$, and the subgroups of $D_n$ isomorphic to $D_n/\{\pm1\}$.

\subsection{Images of endomorphisms of $C_n$, $D_n$, and $A_n$}
\label{subsec: images of endos Cn, Dn, An}

    This section counts the number of possible images of endomorphisms of $C_n$, $D_n$, and $A_n$. If $W_n$ is $C_n$, $D_n$ or $A_n$, and $K\unlhd W_n$, then we compute the number of subgroups of $W_n$ isomorphic to the quotient group $W_n/K$. The normal subgroups and their quotients are described in Theorem \ref{thm: normal subs of Cn, Dn and An}. 

    \begin{lem}
        Suppose $W_n$ is $C_n$, $A_{n-1}$, or $D_n$ where $n\geq2$. The number of involutions in $W_n$ is
        $$ \mathcal{V}(W_n)=-1+\frac{2^{\al(n)}n!}{\lfloor \frac{n}{2}\rfloor!}+\sum_{k=0}^{\lfloor \frac{n}{2}\rfloor-1} \frac{|W_n|}{2^{\tau k}\;k!\;(n-2k)!} $$
        where $\tau=2$ and $\al(n)=\frac{1-(-1)^n}{2}$ when $W_n=C_n$, $\tau=2$ and $\al(n)=0$ when $W_n=D_n$, and $\tau=1$ and $\al(n)=-\lfloor \frac{n}{2}\rfloor$ when $W_n=A_{n-1}$.
        \label{lem: invs in Cn, Dn, An} 
    \end{lem}
    \begin{proof}
         Elements of $C_n$ with involution-type $(l,k)$ are the product of $0\leq k\leq \lfloor \frac{n}{2}\rfloor$ reflections of the form $r_{\pm1,i,j}$ and $0\leq l\leq n-2k$ reflections of the form $r_i$, with all subscripts distinct. Thus, to find the number elements with involution type $(l,k)$, we count the number of ways of choosing $k$ unordered pairs from $n$ objects, multiply this by $2^k$ (since both $r_{+1,i,j}$ and $r_{-1,i,j}$ have signed cycle type $+2$), and choose $l$ objects from the remaining $n-2k$ elements. Hence this number is 
         \begin{align*} 
         1+\mathcal{V}(C_n)
         &=\sum_{k=0}^{\lfloor \frac{n}{2}\rfloor}\frac{2^k}{k!}\binom{n}{2}\dots\binom{n-2(k-1)}{2}\sum_{l=0}^{n-2k}\binom{n-2k}{l} 
         =\sum_{k=0}^{\lfloor \frac{n}{2}\rfloor}\frac{2^nn!}{2^{2k}\;k!\;(n-2k)!}
        \end{align*}
        elements in $C_n $ with order at most $2$. 
        
        An involution of $C_n$ with involution-type $(l,k)$ is contained in $(C_n)_1\cong D_n$ if and only if $l$ is even. For $k<\frac{n}{2}$, there are $2^{n-2k-1}$  even-ordered subsets of $n-2k$ objects, so there are 
        \begin{align*}
            1+\mathcal{V}(D_n)
         &=\sum_{k=0}^{\lfloor \frac{n}{2}\rfloor}\frac{2^k}{k!}\binom{n}{2}\dots\binom{n-2(k-1)}{2}\sum_{l=0}^{\lfloor \frac{n}{2}\rfloor-k}\binom{n-2k}{2l} 
         =\frac{n!}{\lfloor \frac{n}{2}\rfloor!}+ \sum_{k=0}^{\lfloor \frac{n}{2}\rfloor-1}\frac{2^{n-1}n!}{2^{2k}\;k!\;(n-2k)!}
        \end{align*}
        elements in $D_n $ with order at most $2$. 
        
        Involutions of $C_n$ are contained in $S_+\cong A_{n-1}$ if and only if they consist of $1\leq k\leq\lfloor \frac{n}{2}\rfloor$ reflections of the form $r_{+1,i,j}$, so there are
        $$ \mathcal{V}(A_{n-1})=\sum_{k=1}^{\lfloor \frac{n}{2}\rfloor}\frac{1}{k!}\binom{n}{2}\dots\binom{n-2(k-1)}{2} =\sum_{k=1}^{\lfloor \frac{n}{2}\rfloor} \frac{n!}{2^k\;k!\; (n-2k)!}$$
        involutions in $A_{n-1}$.
    \end{proof}

    We aim to find $\mathcal{Z}(C_n,(C_n)_1^+)$, the number of subgroups of $C_n$ isomorphic to $(\FC_2)^2$, but first in Lemma \ref{lem: comm order leq 2} we compute the number of ordered pairs of commuting elements in $C_n$ with order at most $2$. To do this, we fix an arbitrary element $x\in C_n$ with order at most 2 and describe the general form of an element $y\in C_n$ with order at most 2 with which it commutes. The action of $\lb x,y\rb$ on $\RR^n$ induces a partition of the orthonormal basis $e_1,\dots,e_n$ into parts of size $1$, $2$ and $4$. We will count the number of such decompositions, and then the number of pairs of involutions which induce each decomposition. 

    \begin{lem}
        When $n\geq1$, the number of ordered pairs of commuting elements in $C_n$ with order at most $2$ is 
        \begin{equation}
        \sum_{k=0}^{\lfloor \frac{n}{4}\rfloor} \sum_{l=0}^{\lfloor \frac{n}{2}\rfloor-2k} \frac{(\frac{3}{8})^l4^nn!}{2^{7k} \ k!\ l!\ (n-4k-2l)!}      .   
        \label{equ: comm invs}
        \end{equation}
    \label{lem: comm order leq 2}
    \end{lem}
    \begin{proof}
    Fix an element $x$ with order at most $2$ and involution-type $(t,u)$ in $C_n$. Then $x$ has the form 
    $$ x=r_{a_1}\dots r_{a_t}r_{\eps_1,b_1,c_1}\dots r_{\eps_u,b_u,c_u}  $$
    where $\eps_i\in\{\pm1\}$ and $a_i,b_i,c_i\in \{1,\dots,n\}$ are all distinct. Let $\mathrm{Fix}(x)$ be the set of indices of the orthonormal basis vectors of $\RR^n$ which are fixed by the action of $x$ on $\RR^n$. Suppose $y$ is an element of $C_n$ with order at most $2$ and which commutes with $x$. Then $y$ is the product of terms with the following forms:
    \begin{enumerate}[label=(\roman*)]
        \item $r_{\eps_l,b_i,c_i}$ where $\eps_l\in\{\pm1\}$,
        \item $r_{\eps_l,a_i,a_j}$ where $\eps_l\in\{\pm1\}$,
        \item $r_{\eps_l,d,d'}$ where $d,d'\in\mathrm{Fix}(x)$ and $\eps_l\in\{\pm1\}$,
        \item $r_{\eps_k,b_i,b_j}r_{\eps_l,c_i,c_j}$ where $\eps_k,\eps_l\in\{\pm1\}$ so that $\eps_i\eps_j=\eps_k\eps_l$,
        \item $r_{b_i}r_{c_i}$,
        \item $r_{a_i}$,
        \item $r_d$ where $d\in\mathrm{Fix}(x) $.
    \end{enumerate}
    Given $x$ and $y$, one can orthogonally decompose $\RR^n$ into subspaces of dimension $1$, $2$, or $4$ which are preserved under the action of $\lb x,y\rb$ on $\RR^n$, and for which the dimensions of these subspaces are minimal with respect to this property. In particular, these subspaces are spanned by disjoint subsets of the orthonormal basis $\{e_1,\dots,e_n\}$ of $\RR^n$, and so this decomposition induces a partition of $\{1,\dots,n\}$ into parts of size $1$, $2$, or $4$. 
    For example, if $y$ contains a term of type $\mathrm{(iv)}$, then the $4$-dimensional subspace $V=\mathrm{span}\{e_{b_i},e_{b_j},e_{c_i},e_{c_j}\}$, but no proper nontrivial subspace of $V$, is preserved under the action of both $x$ and $y$. In a similar fashion, terms of $y$ with type $\mathrm{(i)}$, $\mathrm{(ii)}$, $\mathrm{(iii)}$, or $\mathrm{(v)}$ beget $2$-dimensional subspaces of $\RR^n$, and terms with type $\mathrm{(vi)}$ or $\mathrm{(vii)}$ give rise to $1$-dimensional subspaces. 
    The subspace spanned by $\{e_{i}: i\in \mathrm{Fix}(x)\cap \mathrm{Fix}(y)\}$ is also spanned by a subset of the orthonormal basis, and so in this minimal decomposition it is broken into $1$-dimensional subspaces spanned by basis vectors. 

    We have shown that a pair of commuting elements with order at most $2$ induces a decomposition of $\RR^n$ into subspaces with dimensions $1$, $2$, or $4$, each of which is spanned by a subset of the orthonormal basis of $\RR^n$. Given such a decomposition of $\RR^n$, we will now count how many pairs $(x,y)$ of commuting elements with order at most $2$ produce it. Suppose $P=\{P_{\al}\}$ is a partition of $\{1,\dots,n\}$ with parts of size 1, 2, or 4.  If $x$ and $y$ are commuting elements of $C_n$ with order at most $2$ which induce $P$, then they are products of terms of the forms listed in Table \ref{tab: terms in commuting refls}. 
    \begin{table}[h] 
    \begin{tabular}{|c|c|c|c|}
        \hline 
         $P_{\alpha}$  & Term in $x$ & Term in $y$ & Coefficient \\
         \hline 
         $\{i\}$ & $r_i^{\ka}$ & $r_i^{\nu}$ & $\ka,\nu\in\{0,1\}$ \\
         \hline 
          & $r_{\eps_1,i,j} $&$r_{\eps_2,i,j}$ & $\eps_1,\eps_2\in\{\pm1\}$\\
         $\{i,j\}$ & $r_{\eps,i,j} $&$(r_ir_j)^{\ka}$ & $\eps\in\{\pm1\},\ka\in\{0,1\}$\\
         & $(r_ir_j)^{\ka} $&$r_{\eps,i,j}$& $\eps\in\{\pm1\},\ka\in\{0,1\}$\\
         \hline 
         $\{i,j,k,l\}$ & $r_{\eps_1,i,j}r_{\eps_1,k,l}$ & $r_{\eps_2,i,l}r_{\eps_2,j,k}$ & $\eps_1,\eps_2\in\{\pm1\}$ \\ 
         & $r_{\eps_1,i,j}r_{-\eps_1,k,l}$ & $r_{\eps_2,i,l}r_{-\eps_2,j,k}$ & $\eps_1,\eps_2\in\{\pm1\}$  \\ 
         \hline 
    \end{tabular}
    \caption{Terms that appear in pairs of commuting elements of $C_n$ with order at most 2.}
    \label{tab: terms in commuting refls}
    \end{table}
    From Table \ref{tab: terms in commuting refls}, we see that for $P_{\al}=\{i\}$ (respectively $P_{\al}=\{i,j\})$, there are $4$ (resp. $12$) ways that $P_{\al}$ can index terms in $x$ and $y$.
    For $P_{\al}=\{i,j,k,l\}$, there are $6$ ways to assign the indices to $x$ and $y$, and $4$ choices of $\eps_1,\eps_2$ for each of the $2$ formats. Hence there are $48$ ways that $P_{\al}$ can index terms in a pair of commuting involutions. 

    The partition $P$ has $0\leq k\leq \lfloor \frac{n}{4}\rfloor$ subsets $P_{\al}\in P$ with $|P_{\al}|=4$, these subsets can be chosen in $\frac{n!}{(4!)^k(n-4k)!k!}$ ways. Similarly,
    $P$ has $0\leq l\leq \lfloor \frac{n}{2}\rfloor-2k$ subsets with $|P_{\al}|=2$ in the remaining $n-4k$ elements and these can be chosen in $\frac{(n-4k)!}{2^l(n-4k-2l)!l!}$ ways. The leftover $n-4k-2l$ elements are `partitioned' into sets of order $1$.
    Hence the total number of ordered pairs of commuting elements with order at most 2 is
    \begin{align*}
        \sum_{k=0}^{\lfloor \frac{n}{4}\rfloor} \sum_{l=0}^{\lfloor \frac{n}{2}\rfloor-2k} 
        \frac{48^kn!}{(4!)^k(n-4k)!k!} \cdot
        \frac{12^l(n-4k)!}{2^l(n-4k-2l)!l!}\cdot
        4^{n-4k-2l}
        =\sum_{k=0}^{\lfloor \frac{n}{4}\rfloor} \sum_{l=0}^{\lfloor \frac{n}{2}\rfloor-2k} \frac{3^l4^nn!}{2^{7k} 8^l\;k!\; l!\;(n-4k-2l)!} .
    \end{align*}
    \end{proof}

    We have computed the number of ordered pairs of commuting elements with order at most 2, but what we require is the number of subgroups of $C_n$ isomorphic to $(\FC_2)^2$. 
    
    \begin{lem}
        For $n\geq2$, the number of subgroups of $C_n$ isomorphic to $(\FC_2)^2$ is 
        $$\mathcal{Z}(C_n,(C_n)_1^+)= \frac{1}{6}\Bigg( 2+\        \sum_{k=0}^{\lfloor \frac{n}{4}\rfloor} \sum_{l=0}^{\lfloor \frac{n}{2}\rfloor-2k} \frac{(\frac{3}{8})^l4^nn!}{2^{7k} \ k!\ l!\ (n-4k-2l)!}         
 -3\sum_{k=0}^{\lfloor \frac{n}{2}\rfloor}\frac{2^nn!}{2^{2k}\ k!\ (n-2k)!}\Bigg) .$$
    \label{lem: subs of Cn isom (Z/2Z)^2}
    \end{lem}
    \begin{proof}
        Subgroups of $C_n$ isomorphic to $(\FC_2)^2$ are generated by pairs of distinct involutions. Thus we remove pairs of the form $(x,x)$, $(x,\id)$, and $(\id,x)$ from Equation \ref{equ: comm invs} by subtracting from it $3(\mathcal{V}(C_n)+1)$. Doing so triple-counts $(\id,\id)$ as an element to remove, so we add $2$. Finally, we divide this by 6, the order of the automorphism group of $(\FC_2)^2$, to account for the six ordered pairs of distinct commuting involutions which generate each subgroup isomorphic to $(\FC_2)^2$. 
    \end{proof}    

    We will count the number of subgroups of $C_n$ which are isomorphic to $\mathfrak{S}_n$ by examining monomorphisms $\mathfrak{S}_n\hookrightarrow C_n$.   Note that $S_+:=\lb r_{+1,1,2},\dots, r_{+1,n-1,n}\rb $ and $S_-:=\lb z_0r_{+1,1,2},\dots, z_0r_{+1,n-1,n}\rb $ are both subgroups of $C_n$ isomorphic to $\FS_n$. 

    \begin{lem}
        For $n\geq 3$ and $n\neq 4$, there are $2^n$ subgroups of $C_n$ isomorphic to $\FS_n$, half of which are conjugate in $C_n$ to $S_+$, and the other half to $S_-$. For $n\geq 5$, the number of subgroups of $D_n$ isomorphic to $\FS_n$ is $2^{n-1}$ when $n$ is odd, and $2^n$ when $n$ is even.
        \label{lem: subgroups isom to Sn in Cn and Dn}
    \end{lem}
    \begin{proof}
    Suppose $n\neq 4$ and let $\varphi$ be a monomorphism $\varphi:\mathfrak{S}_n\hookrightarrow C_n$. By composing $\varphi$ with the projection $\pi:C_n\to \FS_n$, we obtain the endomorphism $f:=\pi\circ\varphi$ of $\mathfrak{S}_n\cong A_{n-1}$. By Theorem \ref{thm: normal subs of Cn, Dn and An}, the map $f$ is either trivial, a sign map (a homomorphism with kernel $\mathfrak{A}_n$ and image isomorphic to $\FC_2$), or an automorphism of $\FS_n$. 
    
    Since $\mathfrak{S}_n$ is not abelian, it cannot inject into $N\cong (\FC_2)^n$ and hence $f$ is nontrivial. Suppose $f$ is a sign map. Then the image of $f$ is generated by some transposition $\tau\in \mathfrak{S}_n$, and hence the image of $\varphi$ is contained inside the subgroup $\lb N,\tau\rb$ of $C_n$. But, since $\varphi$ is injective, $|\im\varphi|=|\mathfrak{S}_n|=n!$ which is larger than $2^{n+1}=|\lb N,\tau\rb|$ for $n\geq 5$. When $n=3$, we can verify by direct computation in the wreath product that there are no injections of $\FS_3$ into $\lb N,\tau\rb$ for any transposition $\tau$. 

    Hence $f$ is an automorphism of $\mathfrak{S}_n$. Using the semidirect product structure of $C_n= N\rtimes S_+$ we define maps $s:\FS_n\to N$ and $t:\FS_n\to \FS_n$ by $\varphi(\sig)=(s(\sig),t(\sig))$. Note that $t\equiv f$ and so by pre-composing $\varphi$ with $f^{-1}$, we can assume that $\varphi(\sig)=(s(\sig),\sig)$. So understanding $\varphi$ reduces to understanding $s$. 
    Let $(i,j)$ and $(k,l)$ be distinct transpositions in $\FS_n$. 
    Since the reflections $r_1,\dots,r_n$ generate $N\cong (\FC_2)^n$ and $s(i,j),s(k,l)\in N$, we can write $s(i,j)=\prod_mr_m^{\eps_m}$ and $s(k,l)=\prod_m r_m^{\ga_m}$ for some coefficients $\eps_m,\ga_m\in\{0,1\}$. We will use the relations of $\FS_n$ to obtain constraints on the coefficients of $s(i,j)$. For example, $\eps_i=\eps_j$ because $s(i,j)^2=\id_N$.
        
    If $\{i,j\}\cap \{k,l\}=\emptyset$, then 
    $$\Big(\prod_{m\neq i,j}r_m^{\eps_m+\ga_m}\Big)r_i^{\eps_i+\ga_j}r_j^{\eps_j+\ga_i}=
    s(i,j)s(k,l)=s(k,l)s(i,j)=
    \Big(\prod_{m\neq k,l}r_m^{\eps_m+\ga_m}\Big)r_k^{\eps_l+\ga_k}r_l^{\eps_k+\ga_l}.$$ 
    By examining the coefficient of $r_k$ we find that $\eps_k=\eps_l$. More generally, the coefficients of any pair of integers disjoint from $\{i,j\}$ are identified so $s(i,j)=(\prod_{m\neq i,j}  r_m)^{\eps}(r_ir_j)^{\eps'}$ for some $\eps,\eps' \in \{0,1\}$, and similarly $s(k,l)=(\prod_{m\neq k,l}  r_m)^{\ga}(r_kr_l)^{\ga'}$ for some $\ga,\ga'\in \{0,1\}$. 
    
    Suppose $|\{k,l\}\cap \{i,j\}|=1$. Without loss of generality, $l=j$ and 
    $$ \Big(\prod_{m\neq i,k}r_m\Big)^{\eps}r_i^{\eps'+\ga'+\ga}r_k^{\ga+\eps'+\ga'}=
    s\big((i,j)^{(j,k)}\big)=s\big((j,k)^{(i,j)}\big) =
    \Big(\prod_{m\neq i,k} r_m\Big)^{\ga}r_i^{\ga'+\eps'+\eps}r_k^{\ga'+\eps'+\eps}.
    $$
    By examining the coefficient of $r_i$ we find that $\ga=\eps$ and so  $s(i,j)=(\prod_{m\neq i,j} r_m)^{\eps}(r_ir_j)^{\eps'}$ and $s(j,k)=(\prod_{m\neq j,k} r_m)^{\eps}(r_jr_k)^{\eps''}$ for some $\eps,\eps',\eps''\in \{0,1\}$. 
    
    In summary, the set of monomorphisms $\FS_n\hookrightarrow C_n$  are in bijection with elements of $C_n\cong(\FC_2)^n\times \mathfrak{S}_n$, with $(\prod r_m^{\eps_m},\sig)\in C_n$ corresponding to the monomorphism defined on the generating set $\{(1,2),\dots,(n-1,n)\}$ of $\FS_n$ by
    $$ \varphi(i,i+1)=\Big(\Big(\prod_{m\neq \sig(i),\sig(i+1)} r_m\Big)^{\eps_1}(r_{\sig(i)}r_{\sig(i+1)})^{\eps_{i+1}},(i,i+1)^{\sig}\Big). $$
    Moreover, since pre-composition of $\varphi$ with an automorphism of $\mathfrak{S}_n$ does not change the setwise image of the $\varphi$, the set of subgroups of $C_n$ isomorphic to $\mathfrak{S}_n$ is the set of images of these monomorphisms. This set corresponds to the choice of $\prod r_m^{\eps_m}\in (\FC_2)^n$, so there are $2^n$ subgroups of $C_n$ isomorphic to $\FS_n$.
    
    Recall that conjugacy classes of elements in $C_n$ correspond to their signed cycle-type. The involution-type of $\varphi(i,i+1)$ is $(0,1)$ if $\eps_1=0$, and $(n-2,1)$ if $\eps_1=1$. Since all transpositions are conjugate in $\FS_n$, the conjugacy class of the image of $\varphi$ is determined by whether $\eps_1$ is $0$ or $1$. Note that the generators of $S_+$ have involution-type $(0,1)$ and the generators of $S_-$ have involution-type $(n-2,1)$, so the $2^n$ subgroups of $C_n$ isomorphic to $\mathfrak{S}_n$ fall into two equally-sized conjugacy classes generated by $S_+$ and $S_-$.

    An involution with involution-type $(t,u)$ is in $(C_n)_1\cong D_n$ if and only if $t$ is odd; and $r_{1,i,j}$ and $z_0r_{+1,i,j}$ have involution-types $(0,1)$ and $(n-2,1)$, respectively. If $n$ is odd conjugates of $S_-$ inside $C_n$ are not contained in $(C_n)_1$, but the conjugates of $S_+$ are. So $D_n$ has $2^{n-1}$ subgroups isomorphic to $\mathfrak{S}_n$ when $n$ is odd. On the other hand, if $n$ is even, then the conjugates of both $S_+$ and $S_-$ are contained in $(C_n)_1$, so $D_n$ contains $2^n$ subgroups isomorphic to $\mathfrak{S}_n$.
    \end{proof}

    Although $S_+$ and $S_-$ are not conjugate in $C_n$, the automorphism group of $C_n$ does act transitively on the set of subgroups of $C_n$ isomorphic to $\FS_n$. The outer automorphism group of $C_n$ has order $2$ when $n$ is odd, and order $4$ when $n$ is even {\rm\cite[p. 27]{AutW_Franzsen}}. These outer automorphisms of $C_n$ map conjugates of $S_+$ to conjugates of $S_-$.  

    The proof of Lemma \ref{lem: subgroups isom to Sn in Cn and Dn} fails when $n=4$ because, as can be seen in the endomorphism table of $A_3\cong \FS_4$ in Table \ref{tab: counting exceptions}, in addition to the trivial map, sign map, and automorphisms, $\FS_4$ also has endomorphisms with kernel the Klein $4$-group, and image isomorphic to $\FS_3$. In the notation of the above proof, when $n=4$, the function $f$ can be either an automorphism or one of these `extra' endomorphisms, the latter of which give rise to $16$ additional subgroups of $C_4$ isomorphic to $\FS_4$, which fall into $4$ equally-sized conjugacy classes. 
    
    In Lemma \ref{lem: subgroups isom to Z/2ZxSn in Cn}, we count the number of subgroups of $C_n$ isomorphic to $\FC_2\times \FS_n$ by examining involutions which centralise $S_{\pm}$.

    \begin{lem}
        When $n\geq2$ and $n\neq 4$, there are $2^{n-1}$ subgroups of $C_n$ isomorphic to $\FC_2\times\FS_n$, all of which are conjugate in $C_n$ to $\lb z_0, S_+\rb$.
        \label{lem: subgroups isom to Z/2ZxSn in Cn}
    \end{lem}
    \begin{proof}
    First, we find subgroups of $C_n$ isomorphic to $\FC_2\times \mathfrak{S}_n$ which contain $S^{+ }$ (respectively  $S_-$). 
    Suppose $n\neq 4$ and $w$ is an involution which centralises $S_+$ (resp. $S_-$). Since $\FS_n$ is centreless, the involution $w$ lies in the abelian subgroup $N\leqslant C_n$, and hence in the centre of $C_n$. Thus $w$ is the unique nontrivial central involution $z_0$ of $C_n$, and so $\lb z_0,S_{+}\rb=\lb z_0,S_-\rb$ is the unique subgroup of $C_n$ isomorphic to $\FC_2\times \FS_n$ which contains $S_{\pm}$.
    
    For any conjugate $H$ of $S_{+}$, the involution $z_0$ centralises but is not contained in $H$ and hence $\lb z_0,H\rb$ is the unique subgroup of $C_n$ which is isomorphic to $\FC_2\times \mathfrak{S}_n$ and which contains $H$. There are $2^{n-1}$ conjugates of $S_+$ and therefore there are $2^{n-1}$ subgroups of $C_n$ isomorphic to $\FC_2\times \mathfrak{S}_n$. 
    \end{proof}

    We will describe the subgroups of $C_n$ isomorphic to $C_n/\{\pm1\}\cong \Inn(C_n)$. Although we don't know the isomorphism type of $C_n/\{\pm1\}$, this group has order $2^{n-1}n!$ and so any subgroup of $C_n$ isomorphic to it has index $2$ in $C_n$, and is therefore normal in $C_n$. As stated in Theorem \ref{thm: normal subs of Cn, Dn and An}, there are $3$ index $2$ subgroups of $C_n$; these are $(C_n)_1$, $\prescript{}{+}{C_n}$, and $C_n^+$, the subgroups of elements with signs $(+1,\pm1)$, $(\pm1,+1)$, and $\pm(+1,+1)$, respectively.

    \begin{lem}
        When $n\geq3$ is odd, there are two subgroups, $(C_n)_1$ and $C_n^+$, of $C_n$ isomorphic to $C_n/\{\pm1\}$. Further, the automorphism group of $C_n/\{\pm1\}$ has order $2^{n-1}n!$.
        When $n\geq4$ is even, there is no subgroup of $C_n$ isomorphic to $C_n/\{\pm1\}$, and no subgroup of $D_n$ isomorphic to $D_n/\{\pm1\}$. 
        \label{lem: index 2 in Cn}
    \end{lem}
    \begin{proof}

    Recall that $(C_n)_1=N^+\rtimes S_+$ is isomorphic to $(\FC_2)^{n-1}\rtimes \FS_n$ and that $z_0$ has involution type $(n,0)$. So $z_0$ lies in $\prescript{}{+}{C_n}$ for any $n$, and is in $(C_n)_1$ and $C_n^+$ if and only if $n$ is even. 

    Suppose $n$ is odd. We will first show that $(C_n)_1$ and $C_n^+$ are isomorphic. The subgroup $S_-$ (respectively $N^+$) is contained in $C_n^+$ since its generators $z_0r_{+1,1,2},\dots,z_0r_{+1,n-1,n}$ (resp. $r_1r_2,\dots,r_{n-1}r_n$) have involution-type $(n-2,1)$ (resp. $(2,0)$) and so sign $(-1,-1)$ (resp. $(1,1)$). Note that $S_-$ and $N^+$ intersect trivially, so $|\lb N^+,S_- \rb |=2^{n-1}n!=|C_n^+| $ and thus $C_n^+=N^+\rtimes S_-$ is isomorphic to $(C_n)_1$.  Observe that $N/\{\pm1\}\cong N^+$ and $S^{\pm}$ intersects trivially with $\{\pm1\}$. So $C_n/\{\pm1\} $ is isomorphic to $(\FC_2)^{n-1}\rtimes \mathfrak{S}_n$. On the other hand, the subgroup $\prescript{}{+}{C_n}$ is isomorphic to $ (\FC_2)^n\rtimes \mathfrak{A}_n$ and not to $(C_n)_1$. Since $C_n/\{\pm1\}$ is isomorphic to $(C_n)_1\cong D_n$, the automorphism group of $C_n/\{\pm1\}$ has order $2^{n-1}n!$ by Theorem \ref{thm: |Aut| for Cn, Dn, An}.

    Suppose $n$ is even. As noted above, when $n$ is even all three of the index $2$ subgroups of $C_n$ contain the central element $z_0$. 
    On the other hand, we will show $C_n/\{\pm1\}$ is centreless. Note that 
    $$ \frac{C_n/\{\pm1\}}{N/\{\pm1\}}\cong \FS_n $$
    so this quotient group is centreless, and hence the centre of $C_n/\{\pm1\}$ lies inside $N/\{\pm1\}$. Suppose $(\prod_mr_m^{\eps_m},\id)\{\pm1\}$ is in the centre of $C_n/\{\pm1\}$. For every $i,j$, since $w\{\pm1\}$ commutes with $r_{+1,i,j}\{\pm1\}$, we know $\eps_i=\eps_j$ and hence the centre of $C_n/\{\pm1\}$ is trivial. 

    Similarly, $D_n^+$ is the only index $2$ subgroup of $D_n$ and it has a nontrivial centre, but $D_n/\{\pm1\}$ is centreless. So there is no subgroup of $D_n$ isomorphic to $D_n/\{\pm1\}$.
    \end{proof}

\subsection{Counting endomorphisms of $C_n$, $D_n$ and $A_n$}
\label{subsec: counting end(Cn)}

    In this section we compute the number of endomorphisms of the irreducible spherical reflection groups of types $C_n$, $D_n$, and $A_n$ in Theorem \ref{thm: H(Cn), H(Dn), H(An)}. Suppose $W_n$ is $C_n$, $D_n$, or $A_n$. Theorem \ref{thm: |Aut| for Cn, Dn, An} on (p. 7) states the order of the automorphism group of $W_n$ and Theorem \ref{thm: normal subs of Cn, Dn and An} (p. 8) lists each normal subgroup $K\unlhd W_n$ and describes the quotient group $W_n/K$. Lemmas   \ref{lem: invs in Cn, Dn, An}- \ref{lem: index 2 in Cn} count the number of subgroups of $W_n$ isomorphic to $W_n/K$ and, where applicable, state the order of $\Aut(W_n/K)$. These data are collated in Table \ref{tab: counting End(Cn), End(Dn), End(An)}.
    \begin{thm}
    Suppose $n\neq 4,6$. Table \ref{tab: counting End(Cn), End(Dn), End(An)} contains the endomorphism tables of $C_n$ and $D_n$. The number of endomorphisms of $C_n$ is
    $$\mathcal{H}(C_n)= 2^{n}n!\Bigg(4+        \sum_{k=0}^{\lfloor \frac{n}{4}\rfloor} \sum_{l=0}^{\lfloor \frac{n}{2}\rfloor-2k} \frac{(\frac{3}{8})^l2^n}{2^{7k} \ k!\ l!\ (n-4k-2l)!}         
\Bigg) ,$$
    and the number of endomorphisms of $D_n$ is 
    $$\mathcal{H}(D_n)=2^{n+\frac{1+(-1)^n}{2}}n!+
        \frac{n!}{\lfloor\frac{n}{2}\rfloor!}+\sum_{k=0}^{\lfloor \frac{n}{2}\rfloor-1} \frac{2^{n-1}n!}{2^{2k}\ k!\ (n-2k)!}.$$
    Suppose $n\neq 3,5$. Table \ref{tab: counting End(Cn), End(Dn), End(An)} contains the endomorphism table of $A_n$. The number of endomorphisms of $A_n$ is 
    $$         \mathcal{H}(A_n)=(n+1)!+\sum_{k=0}^{\lfloor\frac{n+1}{2}\rfloor} 
        \frac{(n+1)!}{2^k\ k!\ (n+1-2k)!} . $$
    \label{thm: H(Cn), H(Dn), H(An)}    
    \end{thm}

    \begin{longtable}{| c | c c c c c c | }
        \hline
             $W$ & $|W:K|$ & $K\unlhd W$ &  $W/K$ &$\mathcal{Z}(W,K)$ & $|\Aut(W/K)|$ & $\mathcal{E}(W,K)$ \\ 
        \hline
             & $1$ & $C_n$& $\{1\}$ & $1$ & $1$ & $1$ \\ 
             
              & $2$ & $C_n^+$ & $\FC_2$ & $\mathcal{V}(C_n)$  & $1$ & $\mathcal{V}(C_n)$ \\
             
              & $2$ & $(C_n)_1 $ &  $\FC_2 $ & $\mathcal{V}(C_n)$ & $1$ & $\mathcal{V}(C_n)$ \\
             
              & $2$ & $\prescript{}{+}{C_n} $ &  $\FC_2 $  & $\mathcal{V}(C_n)$ & $1$ & $\mathcal{V}(C_n)$ \\
             
              & $4$ & $(C_n)_1^+ $ & $(\FC_2)^2 $   & $\mathcal{Z}(C_n,(C_n)_1^+)$& $6$ & $6\cdot\mathcal{Z}(C_n,(C_n)_1^+)$ \\
             
              $C_n$ & $n!$ & $N$  & $\mathfrak{S}_n$ & $2^n $& $n!$ & $2^nn!$\\
             
             & $2 \cdot n!$& $N^+$ & $\FC_2\times \mathfrak{S}_n$  & $2^{n-1}$& $2\cdot n! $& $2^nn!$ \\
            
             & $ 2^{n-1}n! $& $\{\pm1\}$ & $C_n/\{\pm1\}$ & 
            $\begin{cases}
                2 & \text{$n$ odd,}\\ 
                0 & \text{$n$ even}.
            \end{cases}$ &
            $\begin{cases}
                2^{n-1}n! & \text{$n$ odd,}\\
                \mathrm{N/A} & \text{$n$ even.}
            \end{cases}$
            & $\begin{cases}
                2^{n}n! & \text{$n$ odd,}\\
                0 & \text{$n$ even.}
            \end{cases}$\\
            
             & $2^nn!$ & $\{1\}$  &  $C_n$ &1 & $\begin{cases}
                2^nn! & \text{$n$ odd,} \\ 
                2^{n+1}n! & \text{$n$ even.} 
            \end{cases}$ & 
            $\begin{cases}
                2^nn! & \text{$n $ odd,}\\
                2^{n+1}n! & \text{$n$ even.}
            \end{cases}$\\
        \hline 
              & $1$ & $D_n$& $\{1\}$ & $1$ & $1$ & $1$ \\ 

              & $2$ & $D_n^+ $ & $\FC_2 $  & $\mathcal{V}(D_n)$ & $1$ & $\mathcal{V}(D_n)$ \\

             $D_n$ & $n!$ & $N$  & $\mathfrak{S}_n$ & 
             $\begin{cases}
                 2^{n-1} & \text{$n$ odd,}\\
                 2^n & \text{$n$ even.}
             \end{cases}$ 
             & $n!$ & 
             $\begin{cases}
                 2^{n-1}n! & \text{$n$ odd,}\\
                 2^nn! & \text{$n$ even.}\\
             \end{cases}$\\

             & $2^{n-2}n!$ ($n$ even) & $\{\pm1\}$ & $D_n/\{\pm1\}$ & $0$ & $\mathrm{N/A}$ & $0$\\

             & $2^{n-1}n!$ & $\{1\}$  &  $D_n$ & $1$ & $\begin{cases}
                2^{n-1}n! & \text{$n$ odd,} \\ 
                2^{n}n! & \text{$n$ even.}\\
            \end{cases}$ &
            $\begin{cases}
                2^{n-1}n! & \text{$n$ odd,} \\ 
                2^{n}n! & \text{$n$ even.}\\
            \end{cases}$ \\
        \hline 
             & $1$ & $A_n$& $\{1\}$ & $1$ & $1$ & $1$ \\ 

             $A_n$ & $2$ & $A_n^+ $ & $\FC_2 $  & $\mathcal{V}(A_n)$ & $1$ & $\mathcal{V}(A_n)$ \\

            & $(n+1)!$ & $\{1\}$  & $A_n$ & $1$ & $(n+1)!$ & $(n+1)!$\\
        \hline 
        \caption{The endomorphism tables of $C_n$ and $D_n$ when $n\neq 4,6$; and $A_n$ when $n\neq3,5$. See Definition \ref{def: endomorphism table} for a description of an endomorphism table.}
        \label{tab: counting End(Cn), End(Dn), End(An)}
    \end{longtable}

\section{Endomorphisms of exceptional irreducible spherical reflection groups}
\label{sec: exceptional endomorphisms}

    The endomorphism tables and number of endomorphisms for the remaining irreducible spherical reflection groups, $A_3$, $A_5$, $C_4$, $C_6$, $D_4$, $D_6$, $E_6$, $E_7$, $E_8$, $H_3$, $H_4$, and $F_4$, are presented in Table \ref{tab: counting exceptions}. These tables were computed with the assistance of {\sc Magma}.
    Descriptions of the normal subgroups and automorphism groups of spherical reflection groups can be found in {\rm\cite{maxwell_normal_subs}} and {\rm\cite[Ch. 2]{AutW_Franzsen}}, respectively. 

    The only entry in Table \ref{tab: counting exceptions} not calculated by {\sc Magma} is $\mathcal{Z} (E_8,\{\pm1\})$ since the coset table of $E_8/\{\pm1\}$ is too large. 
    Computing $\mathcal{Z} (E_8,\{\pm1\})$ amounts to checking whether $E_8/\{\pm1\}$ is isomorphic to $E_8^+$. 
    The subgroup $E_8^+$ has nontrivial centre $\{\pm1\}$. If $E_8/\{\pm1\}$ had a nontrivial centre, by the correspondence theorem, the centre would contain $E_8^+/\{\pm1\}$, but it is straightforward to verify that this group is non-abelian. Hence these groups are not isomorphic, and so there are no subgroups of $E_8$ isomorphic to $E_8/\{\pm1\}$.  
    
    \begin{longtable}{| c | c c c c c c | c |}
        \hline
             $W$ & $|W:K|$ & $K\unlhd W$ &  $W/K$ &$\mathcal{Z}(W,K)$ & $|\Aut(W/K)|$ & $\mathcal{E}(W,K)$ & $\mathcal{H}(W)$ \\ 
        \hline
            & 1 & $H_3$& $\{1\}$ & $1$& $1$& $1$ & \\ 
        $H_3$ & 2 & $H_3^+$ & $\FC_2$  & $31$& $1$ & $31$ & $272$ \\
            & $60$ & $Z(H_3)$ & $\mathfrak{A}_5$& $1$ & $120$  & $120$  & \\
            & $120$ & $\{1\}$ & $H_3$ & $1$ & $120$ & $120$ & \\
        \hline    
            & $1$ & $H_4$& $\{1\}$& $1$& $1$& $1$ & \\ 
        $H_4$ & $2$ & $H_4^+$ & $\FC_2$ & $571$& $1$ & $571$ & $29372$ \\
            & $7200$ & $Z(H_4)$ & $\Inn(H_4)$ & $0$ & $14400$ & $0$ & \\
            & $14400$ & $\{1\}$ & $H_4$ & $1$ & $28800$ &  $28800$ & \\
        \hline    
            & $1$ & $F_4$& $\{1\}$ & $1$& $1$& $1$ & \\ 
            & $2$ & $F_4^+$ & $\FC_2$  & $139$ & $1$ & $139$ & \\
            & $2$ & $\prescript{}{+}{F_4}$ & $\FC_2$  & $139$ & $1$ & $139$ & \\
            & $2$ & $(F_4)_+$ & $\FC_2$ & $139$ & $1$ & $139$ & \\
            & $4$ & $\prescript{}{+}({F_4})_+$ & $(\FC_2)^2$ & $597$ & $6$ & $3582$ & \\ 
        $F_4$ & $6$ & $(F_4)_1$ & $\mathfrak{S}_3$ & $352$ & $6$ & $2112$ & $30880$ \\  
            & $6$ & $(F_4)_2$ & $\mathfrak{S}_3$ & $352$ & $6$ & $2112$ & \\  
            & $12$ & $(F_4)_1^+$ & $\FC_2\times \mathfrak{S}_3$ & $560$ & $12$ & $6720$ & \\  
            & $12$ & $(F_4)_2^+$ & $\FC_2\times \mathfrak{S}_3$ & $560$ & $12$ & $6720$ & \\ 
            & $36$ & $N$ & $(\mathfrak{S}_3)^2$ & $64$ & $72$ & $4608$ & \\ 
            & $576$ & $\{\pm 1\}$ & $\Inn(F_4)$ & $0$ & $1152$ & $0$ & \\ 
            & $1152$ & $\{1\}$ & $F_4$ & $1$ & $4608$ & $4608$ & \\ 
        \hline   
             & $1$ & $E_6$& $\{1\}$ & $1$ & $1$ & $1$ & \\ 
        $E_6$ & $2$ & $E_6^+$ & $\FC_2$ &  $891$ & $1$ & $891$ & $52732$ \\
             & $51840$ & $\{1\}$ & $E_6$ & 1& $51840$ & $51840$ & \\
        \hline 
             & $1$ & $E_7$& $\{1\}$ & $1$& $1$& $1$ & \\ 
        $E_7$ & 2 & $E_7^+$ & $\FC_2$ & $10207$ & $1$ & $10207$  & $2913248$ \\
             & $1451520$ & $Z(E_7)$ & $\Inn(E_7)$& $1$ & $1451520$  & $1451520$ &  \\
             & $2903040$ & $\{1\}$ & $E_7$& $1$ & $1451520$ & $1451520$ & \\
        \hline 
            & $1$ & $E_8$& $\{1\}$ & $1$& $1$& $1$ & \\ 
        $E_8$ & 2 & $E_8^+$ & $\FC_2$  & $199951$& $1$ & $199951$ & $696929552$ \\
            & $348364800$ & $\{\pm1\}$ & $\Inn(E_8)$ & $0$ & $\mathrm{N/A}$ & $0$ & \\
            & $696729600$ & $\{1\}$ & $E_8$& $1$ & $696729600$  & $696729600$ & \\
        \hline  
                    & $1$ & $A_3$ & $\{1\}$ & $1$ & $1$ & $1$ & \\ 
        $A_3$ & $2$ & $A_3^+ $ & $\FC_2 $ & $9$ & $1$ & $9$ & $58$ \\
            & $6$ & $V_4$ & $\mathfrak{S}_3$ & $4$ & $6$ & $24$ & \\ 
            & $24$ & $\{1\}$  & $A_3$ & $1$ & $24$ & $24$ & \\
        \hline
            & $1$ & $A_5$ & $\{1\}$ & $1$ & $1$ & $1$ & \\ 
        $A_5$ & $2$ & $A_5^+ $ & $\FC_2 $ & $75$ & $1$ & $75$ & $1516$ \\
            & $720$ & $\{1\}$ & $A_5$ & $1$ & $1440$ & $1440$ &  \\
        \hline 
            & $1$ & $C_4$& $\{1\}$ & $1$ & $1$ & $1$ & \\ 
            & $2$ & $C_4^+$ & $\FC_2$ & $75$  & $1$ & $75$ & \\
            & $2$ & $(C_4)_1 $ &  $\FC_2 $ & $75$ & $1$ & $75$ & \\
            & $2$ & $\prescript{}{+}{C_4} $ &  $\FC_2 $  & $75$ & $1$ & $75$ & \\
            & $4$ & $(C_4)_1^+ $ & $(\FC_2)^2$ & $277$ & $6$ & $1662$ & \\
        $C_4$ & $6$ & $N\rtimes V_4 $ & $\mathfrak{S}_3$ & $64$ & $6$ & $384$ & $6496$ \\
            & $12$ & $N^+\rtimes V_4 $ & $\FC_2\times \mathfrak{S}_3$ & $96$ & $12$ & $1152$ & \\
            &  $24$ & $N$  & $\mathfrak{S}_4$ & $32$ & $24$ & $768$ & \\
            & $48$ & $N^+$ & $\FC_2\times \mathfrak{S}_4$  & $32$& $48$ & $1536$ & \\
            & $192$ & $\{\pm1\}$ & $C_4/\{\pm1\}$ & $0$ & $384$ & $0$ & \\ 
            & $384$ &  $\{1\}$  &  $C_4$ & $1$ & $768$ & $768$ & \\ 
        \hline 
            & $1$ & $C_6$& $\{1\}$ & $1$ & $1$ & $1$ & \\ 
            & $2$ & $C_6^+$ & $\FC_2$ & $1383$  & $1$ & $1383$ & \\
            & $2$ & $(C_6)_1 $ &  $\FC_2 $ & $1383$ & $1$ & $1383$ & \\
            & $2$ & $\prescript{}{+}{C_6} $ &  $\FC_2 $  & $1383$ & $1$ & $1383$ & \\
        $C_6$ & $4$ & $(C_6)_1^+ $ & $(\FC_2)^2$ & $32631$ & $6$ & $195786$ & $476416$ \\
            & $720$ & $N$  & $\mathfrak{S}_6$ & $64$ & $1440$ & $92160$ & \\
            & $1440$ & $N^+$ & $\FC_2\times \mathfrak{S}_6$  & $32$& $2880$ & $92160$ & \\
            & $192$ & $\{\pm1\}$ & $C_6/\{\pm1\}$ & $0$ & $23040$ & $0$ & \\ 
            & $384$ &  $\{1\}$  &  $C_6$ & $1$ & $92160$ & $92160$ & \\ 
        \hline 
            & $1$ & $D_4$& $\{1\}$ & $1$ & $1$ & $1$ & \\ 
            & $2$ & $D_4^+$ & $\FC_2$ & $43$  & $1$ & $43$ & \\
            & $6$ & $N\rtimes V_4 $ & $\mathfrak{S}_3$ & $32$ & $6$ & $192$ & \\
        $D_4$ & $24$ & $N$ & $ \mathfrak{S}_4$  & $24$ & $24$ & $576$ & $3116$ \\
            & $24$ & $(D_4)_{13}$ & $ \mathfrak{S}_4$  & $24$ & $24$ & $576$ & \\
            & $24$ & $(D_4)_{14}$ & $ \mathfrak{S}_4$  & $24$ & $24$ & $576$ & \\
            & $96$ & $\{\pm1\}$ & $D_4/\{\pm1\}$ & $0$ & $576$ & $0$ & \\ 
            & $192$ &  $\{1\}$  &  $D_4$ & $1$ & $1152$ & $1152$ & \\ 
        \hline 
            & $1$ & $D_6$& $\{1\}$ & $1$ & $1$ & $1$ & \\ 
            & $2$ & $D_6^+$ & $\FC_2$ & $751$  & $1$ & $751$ & \\
        $D_6$ & $720$ & $N$ & $ \mathfrak{S}_6$  & $64$ & $1440$ & $92160$ & $138992$ \\
            & $11520$ & $\{\pm1\}$ & $D_6/\{\pm1\}$ & $0$ & $23040$ & $0$ & \\ 
            & $23040$ &  $\{1\}$  &  $D_6$ & $1$ & $46080$ & $46080$ & \\ 
        \hline 
        \caption{The endomorphism tables of exceptional irreducible spherical reflection groups. See Definition \ref{def: endomorphism table} for a description of an endomorphism table.}
        \label{tab: counting exceptions}
    \end{longtable}

\section{Homomorphisms between spherical reflection groups}
\label{sec: Hom(W,W')}

    In this section we compute the number of homomorphisms between some irreducible spherical reflection groups using the endomorphism tables of these groups.

    By examining Table \ref{tab: counting End(I2(m))}, the endomorphism table of $I_2(m)$, we obtain the following Corollary to Proposition \ref{prop: H(I2(m))}.
    \begin{cor}
        The number of homomorphisms from the dihedral group of order $2l$ to the dihedral group of order $2m$ is
        $$ 
        |\Hom(I_2(l),I_2(m))|=
        \begin{cases}
            4+4m+m\gcd(l,m) & \text{if $l$ and $m$ are both even,}\\
            1+2m+m\gcd(l,m) & \text{if $l$ is even and $m$ is odd,}\\
            2+m\gcd(l,m) & \text{if $l$ is odd and $m$ is even,}\\
            1+m\gcd(l,m) & \text{if $l$ and $m$ are both odd.}
        \end{cases}
        $$
    \end{cor}

    We will now describe the homomorphisms from $I_2(p)$, the dihedral group of order $2p$ where $p$ is an odd prime, to $\FS_n\cong A_{n-1}$ and $C_n$. For any group $G$, Table \ref{tab: counting Hom(I2(p),G)} is the homomorphism table of $(I_2(p),G)$. So computing the order of $\Hom(I_2(p),G)$ amounts to counting the number of involutions in $G$, and the number of subgroups of $G$ isomorphic to $I_2(p)$.
     \begin{table}[h]
        \begin{tabular}{ c c c c c c}
        \hline
             $|I_2(p):K|$ & $K\unlhd I_2(p)$ &  $I_2(p)/K$ & $\mathcal{Z}(I_2(p),G,K)$ & $|\Aut(I_2(p)/K)|$ & $\mathcal{E}(I_2(p),G,K)$ \\ 
        \hline 
             $1$ & $I_2(p)$& $\{1\}$ & $1$ & $1$ & $1$ \\ 

             $2$ & $\lb r\rb$ & $\FC_2 $  & $\mathcal{V}(G)$ & $1$ & $\mathcal{V}(G)$ \\

            $2p$ & $\{1\}$  & $I_2(p)$ & $\mathcal{Z}(I_2(p),G,\{1\})$ & $p(p-1)$ & $p(p-1)\mathcal{Z}(I_2(p),G,\{1\})$\\
        \hline 
        \end{tabular}
        \caption{The homomorphism table for $\Hom(I_2(p),G)$ where $p$ is an odd prime and $G$ is any group. See Definition \ref{def: endomorphism table} for a description of an homomorphism table.}
        \label{tab: counting Hom(I2(p),G)}
    \end{table}
    
    \begin{prop}
        Suppose $W_n$ is $C_n$ or $A_{n-1}\cong\FS_n$, and $p$ is an odd prime with $p\leq n$. Then the number of subgroups of $W_n$ isomorphic to the dihedral group of order $2p$ is
        \begin{align}
        &\mathcal{Z}(I_2(p),W_n,\{1\})=\frac{|W_n|}{p(p-1)}
            \sum_{k=1}^{\lfloor\frac{n}{p}\rfloor}
            \sum_{l=0}^{\lfloor\frac{k}{2}\rfloor}
            \sum_{m=0}^{\lfloor\frac{n-kp}{2}\rfloor}
            \frac{1}{(k-2l)!\;l!\;p^l\;2^{\tau(l+m)}\;(n-kp-2m)!\; m!}
            \label{equ: Subs of I2(p) in Cn}
        \end{align}
        where $\tau=1$ when $W_n=\FS_n$, and $\tau=2$ when $W_n=C_n$.
        \label{prop: count I2(p) in Sn and Cn}
    \end{prop}
    \begin{proof}
    Let $W_n$ be $C_n$ or $A_{n-1}$. Subgroups which are isomorphic to $I_2(p)$ are generated by an order $p$ element $\sig$ and an involution $x$ satisfying $(\sig x)^2=1$. 
    Suppose $\sig$ and $x$ are such a pair in $C_n$. Then $\sig$ has the form
    $$ \sig=r_{\eps_{1,1},a_{1,1},a_{1,2}}\dots r_{\eps_{1,p-1},a_{1,p-1},a_{1,p}}\dots 
    r_{\eps_{k,1},a_{k,1},a_{k,2}}\dots r_{\eps_{k,p-1},a_{k,p-1},a_{k,p}}$$
    where $a_{i,j}\in \{1,\dots,n\}$ are all distinct, and $\eps_{i,j}\in\{\pm1\}$. Let $U_i=\{a_{i,1},\dots,a_{i,p}\}$, $U=\cup_{i=1}^kU_i$, and $T$ be the set of coordinate indices permuted by $x$. The order $p$ element $\sig$ lies in $S_+$ if and only if $\eps_{i,j}=1$ for all $i,j$. Hence the number of order $p$ elements in $W_n$ is
    $$ \sum_{k=1}^{\lfloor \frac{n}{p}\rfloor} \binom{n}{p}\dots \binom{n-(k-1)p}{p} \frac{(2^{\la(p-1)}(p-1)!)^k}{k!}= \sum_{k=1}^{\lfloor \frac{n}{p}\rfloor} \frac{2^{\la k(p-1)}n!}{p^k\;(n-kp)!\;k!}$$
    where $\la=1$ for $C_n$ and $\la=0$ for $A_{n-1}$.

    If $U_i\subseteq T$, then $x$ contains $r_{\eps,a_{i,p},a_{j,m}}$ for some  $i\neq j$, $m\in \{1,\dots,p\}$, and $\eps\in\{\pm1\}$. The relation $(\sig x)^2=1$ implies that $x$ contains
    \begin{equation*}
    \prod_{l=1}^{p} r_{\eps \eps_{i,p-l}\dots \eps_{i,p-1}\eps_{j,m}\dots\eps_{j,m+(l-1)},a_{i,p-l},a_{i,m+l}}.
    \end{equation*}
    where addition in the second subscript of the coefficients $a_{j,m}$ is performed modulo $p$.
    Note that this implies $U_i\cup U_j\subseteq T$. There are $l$ pairs $\{U_i,U_j\}$ relating to terms of this form, where $0\leq l\leq \lfloor \frac{k}{2}\rfloor$ . There are $\frac{k!}{2^ll!(k-2l)!}$ ways to choose $l$ unordered pairs $\{U_i,U_j\}$ from $U_1,\dots,U_k$, and for each of these $l$ pairs there are $2p$ choices for $\eps\in\{\pm1\}$ and $m\in\{1,\dots,p\}$ (in $S_+$ there are $p$ choices since $\eps=1$). 
    
    The remaining $k-2l$ sets $U_i$ are not contained in $T$ so there is some $m\in\{1,\dots,p\}$ such that $x:a_{i,m}\mapsto \eps a_{i,m}$ where $\eps\in \{\pm1\}$. The relation $(x\sig)^2=1$ implies that $x$ contains  
    \begin{equation}
    \prod_{l=1}^{(p-1)/2} r_{\eps \eps_{i,m-l}\dots \eps_{i,m+l-1},a_{i,m-l},a_{i,m+l}}
    \label{equ: type 1 I2(p) in Sn}
    \end{equation}
    In particular, note that $U_i\setminus \{a_{i,m}\}$ is contained in $T$. There are $(2p)^{k-2l}$ choices for $U_i$, $\eps$, and $m$ (or $p^{k-2l}$ choices in $S_+$ where $\eps=1)$.

    The complement of the support of $\sig$ has order $n-kp$, and so $x$ contains $0\leq m\leq \lfloor\frac{n-kp}{2}\rfloor$ terms $r_{\eps,i,j}$ where $\eps\in\{\pm1\}$ with support disjoint from $U$. There are $\frac{(n-kp)!}{(n-kp-2m)!\;m!}$ ways to choose $m$ unordered pairs from these $n-kp$ elements and $\eps\in\{\pm1\}$ for each pair (there are $\frac{(n-kp)!}{2^m\;(n-kp-2m)!\;m!}$ choices in $S_+$ where $\eps=1$).
    
    The remaining $n-kp-2m$ indices appear as terms with the form $r_{i}^{t}$ where $t\in\{0,1\}$ (in $S_+$, the remaining indices are in the fixed set of $x$). 

    A subgroup of $W_n$ isomorphic to $I_2(p)$ contains $p(p-1)$  generating pairs $(\sig,x)$, so the number of such subgroups is 
    $$ \frac{1}{p(p-1)}
        \sum_{k=1}^{\lfloor\frac{n}{p}\rfloor}
        \frac{2^{\la k(p-1)}n!}{p^k\ (n-kp)!\ k!}
        \sum_{l=0}^{\lfloor\frac{k}{2}\rfloor}
        \frac{2^{\la l}p^l\ k!}{ 2^l\ l!\  (k-2l)!} 
        (2^{\la}p)^{k-2l}
        \sum_{m=0}^{\lfloor\frac{n-kp}{2}\rfloor}
        \frac{2^{\la m}\ (n-kp)!}{2^{m}\ (n-kp-2m)!\ m!}2^{\la(n-kp-2m)}
         $$
    where $\la=1$ when $W_n=C_n$, and $\la=0$ when $W_n=\FS_n$.
    \end{proof}

    We have already computed the number of involutions in $C_n$ and $\FS_n$, so we can now prove Theorem \ref{thm: hom(I2(p),W)}.

   \begin{proof}[Proof of Theorem \ref{thm: hom(I2(p),W)}]
        Note that for $W_n=C_n$ or $\FS_n$, if the formula in Equation \ref{equ: Subs of I2(p) in Cn} were extended to $k=0$, then the $k=0$ summand would be equal to $\mathcal{V}(W_n)+1$ (which is computed in Lemma \ref{lem: invs in Cn, Dn, An}). 
        From Table \ref{tab: counting Hom(I2(p),G)}, we see that the number of homomorphisms from $I_2(p)$ to a group $W_n$ is determined by $\mathcal{V}(W_n)$ and the number of subgroups of $W_n$ isomorphic to $I_2(p)$. When $p>n$, there are no subgroups of $W_n$ isomorphic to $I_2(p)$ by Lagrange's theorem, so the number of homomorphisms $I_2(p)\to W_n$ is $\mathcal{V}(W_n)+1$, and since $0\leq k\leq \lfloor \frac{n}{p}\rfloor=0$, the claim holds. When $p\leq n$, the number of subgroups of $W_n$ isomorphic to $I_2(p)$ is given by Proposition \ref{prop: count I2(p) in Sn and Cn}.
   \end{proof}

\section{Random endomorphisms}
\label{sec: random endos}

    In this section we describe properties of random endomorphisms of $C_n$. We are interested in examining the behaviour of $\End(C_n)$ as $n\to \infty$ so, to avoid the exceptional endomorphisms of $C_n$ that occur when $n=4$ and $n=6$, we assume for the remainder of this section that $n\geq 7.$
    Let $X_n$ be the order of the image of an endomorphism drawn uniformly at random from $\End(C_n)$. Then the probability that a random endomorphism $\varphi\in \End(C_n)$ satisfies $|\varphi(C_n)|=k$ is 
    \begin{equation}
        \PP\big[X_n=k\big]=
        \frac{1}{\mathcal{H}(C_n)}\sum_{K\unlhd C_n \text{ with } |C_n:K|=k} \mathcal{E}(C_n,K)     .
    \label{equ: random variable general}
    \end{equation} 
    The endomorphism table of $C_n$ (shown in Table \ref{tab: counting End(Cn), End(Dn), End(An)}) lists, in order of increasing index, the normal subgroups $K$ of $C_n$ along with $\mathcal{E}(C_n,K)$, the number of endomorphisms of $C_n$ with kernel $K$, and so this table allows us to compute $\PP[X_n=k]$. 
    \begin{equation}
     \PP\big[X_n=k\big]\cdot \mathcal{H}(C_n)=
    \begin{cases}
        1 & k=1,\\
        3\mathcal{V}(C_n) & k=2,\\
        6\mathcal{Z}(C_n,(C_n)_1^+) & k=4,\\
        2^nn! & k=n! \text{ or } 2\cdot n!,\\
        2^nn! & \text{$n$ is odd and $k=2^{n-1}n!$ or $2^nn!$},\\
        2^{n+1}n! & \text{$n$ is even and $k=2^nn!$},\\ 
        0 & \text{else.}
    \end{cases}
    \label{equ: random variable Cn}
    \end{equation}

    We will examine two sequences which are closely related to 
    $\mathcal{V}(C_n)$ and $\mathcal{Z}(C_n,(C_n)_1^+)$, the number of subgroups of $C_n$ isomorphic to $\FC_2$ and $(\FC_2)^2$, respectively; and $\mathcal{H}(C_n)$, the number of endomorphisms of $C_n$. 
        \begin{defn} Let $a_n$ and $b_n$ be the sequences defined by:
            $$ a_n=\sum_{k=0}^{\lfloor \frac{n}{2}\rfloor}\frac{1}{2^{2k}\ k! \ (n-2k)!}
        \quad \text{and} \quad 
        b_n=\sum_{k=0}^{\lfloor \frac{n}{4}\rfloor} \sum_{l=0}^{\lfloor \frac{n}{2}\rfloor-2k} \frac{(\frac{3}{8})^l2^n}{2^{7k} \ k!\ l!\ (n-4k-2l)!}.$$
        \end{defn}
    By Lemmas \ref{lem: invs in Cn, Dn, An} and \ref{lem: subs of Cn isom (Z/2Z)^2}, the numbers of subgroups of $C_n$ isomorphic to $\FC_2$ and $(\FC_2)^2$ are
    \begin{equation}  
    \mathcal{V}(C_n)=-1+2^nn! \ a_n \quad \text{and} \quad 
    \mathcal{Z}(C_n,(C_n)_1^+)=\frac{1}{6}\big( 2-3\cdot 2^nn!\ a_n +2^nn! \ b_n\big).
    \label{equ: seq of V(Cn) and Lambda(Cn,Cn1+)}
    \end{equation}
    Moreover, from Theorem \ref{thm: H(Cn), H(Dn), H(An)} we see that the number of endomorphisms of $C_n$ is 
    \begin{equation}
        \mathcal{H}(C_n)=2^{n+2}n!+2^nn!\ b_n.
        \label{equ: ends of Cn sequence}
    \end{equation}
    \begin{lem}
        The sequences $(a_n)$ and $(b_n)$ tend to $0$ as $n$ tends to infinity, 
        $$\lim_{n\to \infty}a_n=\lim_{n\to \infty}b_n=0.$$
        \label{lem: annoying limits}
    \end{lem}
    \begin{proof}
        At least one of $k$, $l$ and $n-4k-2l$ is at least $\frac{n}{7}$. So 
        $(k!\ l!\ (n-4k-2l)! )\geq \lfloor \frac{n}{7}\rfloor!$ and hence
        \begin{align*}
            b_n& \leq        
            \sum_{k=0}^{\lfloor \frac{n}{4}\rfloor} \sum_{l=0}^{\lfloor \frac{n}{2}\rfloor-2k} \frac{2^n}{ k!\ l!\ (n-4k-2l)!}\leq
            \sum_{k=0}^{\lfloor \frac{n}{4}\rfloor} \sum_{l=0}^{\lfloor \frac{n}{2}\rfloor-2k} \frac{2^n}{\lfloor \frac{n}{7}\rfloor!}\leq 
            \frac{2^nn^2}{\lfloor \frac{n}{7}\rfloor!}\xrightarrow{n\to \infty}0.
        \end{align*}
        Since $0\leq a_n\leq b_n$ (this can be seen by examining the $k=0$ summand in the definition of $b_n$), $a_n$ also tends to $0$ as $n$ tends to $\infty$.
    \end{proof}

    \begin{thm}
    As $n$ tends to infinity, a random endomorphism of $C_n$ almost surely has large image and small kernel:
    \begin{enumerate}[label=(\alph*)]
        \item The order of the image of an endomorphism of $C_n$ is asymptotically almost surely between $n!$ and $2^nn!=|C_n|$, 
        \begin{align*}
        \lim_{n\to \infty} \PP\big
        [ n! \leq X_n\leq 2^nn!\big]
        & = 1.
        \end{align*}
        
        \item The number of endomorphisms of $C_n$ is asymptotically $2^{n+2}n!$,
        $$ \mathcal{H}(C_n)\sim 2^{n+2}n! $$
    \end{enumerate}
    \label{thm: almost all ends of Cn have small kernels}
    \end{thm}
    \begin{proof}
        As noted in Equation \ref{equ: ends of Cn sequence}, the number of endomorphisms of $C_n$ is
        $$\mathcal{H}(C_n)=2^{n+2}n!+2^{n}n!\ b_n\geq 2^{n+2}n!$$
        From Equation \ref{equ: random variable Cn}, we see that the number of endomorphisms of $C_n$ whose image has order less than $n!$ is $1+3\mathcal{V}(C_n)+6\mathcal{Z}(C_n,(C_n)_1^+)$ which, by Equation \ref{equ: seq of V(Cn) and Lambda(Cn,Cn1+)}, is equal to $2^{n}n!\ b_n$. Hence Lemma \ref{lem: annoying limits} implies that 
        $$ \PP[X_n < n!] =\frac{2^{n}n!\ b_n}{\mathcal{H}(C_n)}\leq \frac{2^{n}n!\ b_n}{2^{n+2}n!}\xrightarrow{n\to \infty}0.$$
        Since the order of the image of an endomorphism of $C_n$ lies between $1$ and $2^nn!=|C_n|$, this proves part (a).
        
        For part (b), observe that
        $$ \lim_{n\to \infty}\frac{\mathcal{H}(C_n)}{2^{n+2}n!}=
        \lim_{n\to \infty}\frac{2^{n+2}n!+2^{n}n!\ b_n}{2^{n+2}n!}=1. $$
    \end{proof}

    Equation \ref{equ: random variable Cn} lists $\PP[X_n=k]$ for any $k$, so we can compute the expected value of the order of an endomorphism of $C_n$.

    \begin{proof}[Proof of Theorem \ref{thm: expected kernel-index of Cn, Dn, An}]
        The order of the image of an endomorphism of $C_n$ lies between $1$ and $2^nn!$ so the expectation of the order of the image of a random endomorphism of $C_n$ is 
        $$ 
        \EE[X_{n}]=\sum_{k=1}^{2^nn!} k\PP\big[ X_{n}= k\big] 
        =\frac{1}{\mathcal{H}(C_n)}\sum_{k=1}^{2^nn!}k\sum_{K\unlhd C_n, |C_n:K|=k} \mathcal{E}(C_n,K).
        $$
        For odd $n$, Equation \ref{equ: random variable Cn} implies that the expectation is
        $$\EE[X_{2n+1}]= \frac{1}{\mathcal{H}(C_{2n+1})}\Big(1 + 6\mathcal{V}(C_{2n+1} )+24 \mathcal{Z}(C_{2n+1} ,(C_{2n+1} )_1^+)+ 3\cdot 2^{2n+1}(2n+1)!^2 (1 +2^{2n})\Big).
        $$
        Using Equations \ref{equ: seq of V(Cn) and Lambda(Cn,Cn1+)} and \ref{equ: ends of Cn sequence} and Lemma \ref{lem: annoying limits}, we find that
        $$ 
        \frac{\EE[X_{2n+1}]}{3\cdot 2^{(2n+1)-3}(2n+1)!}=
        \frac{3+ 2^{2n-1}(2n+1)!\big[
        16\ b_{2n+1}-24 \ a_{2n+1}+12(2n+1)!(1+2^{2n})\big]}{3\cdot 2^{4n-3} (2n+1)!^2(4+b_{2n+1})}
        \xrightarrow{n\to \infty}1.
        $$
        Similar arguments apply to $C_{2n}$, $D_n$, and $A_{n}$.
    \end{proof}

    Since endomorphisms of $C_n$ with image-order $2^nn!$ are automorphisms, by computing $\PP[X_n=2^nn!]$ we find the probability that a random $\varphi\in \End(C_n)$ lies in $\Aut(C_n)$.

    \begin{proof}[Proof of Theorem \ref{thm: prob of automorphism Cn, Dn, An}]
        From Table \ref{tab: counting End(Cn), End(Dn), End(An)}, we see that there are $2^{2n+1}(2n)!$ automorphisms of $C_{2n}$ and by Theorem \ref{thm: almost all ends of Cn have small kernels} part (b), the number of endomorphisms of $C_{2n}$ is asymptotically $2^{2n+2}(2n)!$. Hence,
        $$ \PP\big[ X_{2n}=2^{2n}(2n)!]= \frac{2^{2n+1}(2n)!}{\mathcal{H}(C_{2n})}\sim \frac{2^{2n+1}(2n)!}{2^{2n+2}(2n)!}=\frac{1}{2}. $$
        Similar arguments apply to $C_{2n+1}$, $D_n$, and $A_{n-1}$.
    \end{proof}

    Aside from examining the order of the image of a random endomorphism, we can also look at other properties of endomorphisms. Since we have classified endomorphisms of $C_n$ primarily in terms of their kernels, properties relating to the kernel are particular straightforward to investigate. 
    
    \begin{proof}[Proof of Theorem \ref{thm: central endo}]
        The involution $z_0=r_1\dots r_n$ which generates the centre of $C_n$ has involution-type $(n,0)$. 
        Since $z_0$ lies in every nontrivial normal subgroup of $C_{2n}$, the probability that $z_0$ lies in the kernel of a random endomorphism of $C_{2n}$ is 
        $$ 1-\PP[X_{2n}=2^{2n}(2n)!]=1-\frac{2^{2n+1}(2n)!}{\mathcal{H}(C_{2n})}\sim \frac{1}{2}. $$
        On the other hand, the normal subgroups of $C_{2n+1}$ which contain $z_0$ are $C_{2n+1}$, $\prescript{}{+}{C_{2n+1}}$, $N$, and $\{\pm1\}$. So the probability that $z_0$ lies in the kernel of a random endomorphism of $C_{2n+1}$ is 
        $$ 
        \frac{1+\mathcal{V}(C_{2n+1})+2^{2n+2}(2n+1)!}{\mathcal{H}(C_{2n+1})}=\frac{2^{2n+1}(2n+1)!\big(a_{2n+1}+2\big)}{\mathcal{H}(C_{2n+1})}\sim \frac{1}{2}.
        $$
    \end{proof}

    There are two conjugacy classes of reflections in $C_n$: the class consisting of the $n$ coordinate-inverting reflections conjugate to $r_1$, and the class of the $n(n-1)$ coordinate-permuting reflections conjugate to $r_{+1,1,2}$ (see Section \ref{sec: end(Cn), end(Dn), end(An)} for a more detailed explanation).
    We can ask if these conjugacy classes are equally likely to lie in the kernel of a random endomorphism of $C_n$. 
    
    \begin{cor}
        The limit of the probability that a reflection conjugate to $r_{+1,1,2}$ lies in the kernel of a random endomorphism of $C_n$ is $0$, while the limit of the probability that a reflection conjugate to $r_1$ lies in the kernel of a random endomorphism of $C_n$ is $\frac{1}{4}$. If $r$ is a reflection in $C_n$, then
        \begin{align*}
            \lim_{n\to\infty} \PP_{\varphi\in \End(C_n)}\big[ 
            r\in \ker(\varphi)
            \big]= 
            \begin{cases}
                0 & \text{if $r\sim {r_{+1,1,2}}$},\\
                 \frac{1}{4} & \text{if $r\sim r_1$.}
            \end{cases}
        \end{align*}
    \end{cor}
    \begin{proof}
        The reflection $r_{i}$ lies in $C_n$, $\prescript{}{+}{C_n}$, and $N$ so the probability it is in the kernel is 
        $$ \lim_{n\to\infty} \frac{1+\mathcal{V}(C_n)+2^nn!}{\mathcal{H}(C_n)}= \frac{1}{4} .$$
        The reflection $r_{\pm1,i,j}$ lies in $C_n$ and $(C_n)_1$ so the probability it is in the kernel is 
        $$ \lim_{n\to\infty}\frac{1+\mathcal{V}(C_n)}{\mathcal{H}(C_n)}= 0 .$$
    \end{proof}

    Since we know the normal subgroups $K\unlhd C_n$ and the isomorphism types of $C_n/K$, it is a straightforward task to check whether the image of an endomorphism is normal in $C_n$. 

    \begin{thm}
        The limit as $n$ tends to infinity of the probability that the image of a random endomorphism of $C_n$ is normal in $C_n$ is one half.
        \begin{align*}
            \lim_{n\to\infty} \PP_{\varphi\in \End(C_n)}\big[ 
            \im(\varphi)\unlhd C_n
            \big]= \frac{1}{2}
        \end{align*}
    \end{thm}
    \begin{proof}
        There are three index $2$ subgroups of $C_n$ and $\{\pm1\}$ is the unique order $2$ normal subgroup of $C_n$. Therefore there are $3$ endomorphisms of $C_n$ whose image has order $2$ and is normal in $C_n$. There are no normal subgroups of $C_n$ with orders $4$, $n!$, or $2\cdot n!$ and so no endomorphisms of $C_n$ with normal images of these orders. If $n$ is odd, each of the $2^nn!$ endomorphisms of $C_n$ with kernel $\{\pm1\}$ have image normal in $C_n$, but if $n$ is even there is no endomorphism of $C_n$ with kernel $\{\pm1\}$. The images of endomorphisms with kernel $\{1\}$ or $C_n$ are normal in $C_n$, and there are $2^nn!+1$ of these when $n$ is odd, and $2^{n+1}n!+1$ when $n $ is even. Hence we find that
        $$\lim_{n\to\infty} \PP_{\varphi\in \End(C_n)}\big[ 
            \im(\varphi)\unlhd C_n
            \big]= 
            \lim_{n\to\infty} \frac{2^{n+1}n!+4}{\mathcal{H}(C_n)}=
            \frac{1}{2}.$$
    \end{proof}
    
\printbibliography[title={References}]

\end{document}